\documentclass[a4j,11pt, leqno]{amsart}
\usepackage{epstopdf}
\usepackage{amsmath,amssymb,amscd, amsthm}
\usepackage{epic,eepic,epsfig}
\usepackage{comment} 
\usepackage{mathrsfs,upgreek}
\usepackage[all,cmtip]{xy}
\usepackage{graphicx}
\usepackage{array} 
\usepackage{caption}
\usepackage{here}
\usepackage{enumerate}  
\usepackage{enumitem}
\usepackage{tikz-cd}
 \usepackage[nodayofweek]{datetime}

\makeatletter
\renewcommand{\theenumi}{{\roman{enumi}}}
\def\@settitle{\begin{center}%
  \baselineskip14\p@\relax
  \normalfont\LARGE\bfseries
  \@title
  \ifx\@subtitle\@empty\else
     \\[1ex] 
     \normalsize\mdseries\@subtitle
  \fi
 \ifx\@didication\@empty\else
     \\[2ex] 
     \large\mdseries\it\@dedication
  \fi
  \end{center}%
}
\def\subtitle#1{\gdef\@subtitle{#1}}
\def\@subtitle{}
\def\dedication#1{\gdef\@dedication{#1}}
\def\@dedication{}
\makeatother

\makeatletter
\def\vmargin@#1#2#3{
\setbox0=\hbox{#3}%
\rule[#1]{0pt}{\ht0}%
\lower\dp0\hbox{\rule[-#2]{0pt}{\dp0}}%
\box0%
}
\def\vmargin#1#2#3{
\mathchoice
{\vmargin@{#1}{#2}{$\displaystyle #3$}}
{\vmargin@{#1}{#2}{$\textstyle #3$}}
{\vmargin@{#1}{#2}{$\scriptstyle #3$}}
{\vmargin@{#1}{#2}{$\scriptscriptstyle #3$}}
}
\makeatother

\usepackage{wrapfig}

\makeatletter
\renewcommand{\section}{\@startsection
{section}{1}{0mm}{5mm}{2mm}{\raggedright\bfseries}}

 \@addtoreset{equation}{section}
\makeatother

\newtheorem{Lemma}[equation]{Lemma}

\newtheorem{Corollary}[equation]{Corollary}
\newtheorem{Proposition}[equation]{Proposition}
\theoremstyle{definition}
\newtheorem{Definition}[equation]{Definition}
\newtheorem{Construction}[equation]{Construction}
\newtheorem{Example}[equation]{Example}

\theoremstyle{remark}
\newtheorem{Note}[equation]{Note}

\begin{document}
\newcommand{\ig}[1]{\includegraphics[width=0.25\hsize]{#1}}
\newcommand{\igg}[1]{\includegraphics[width=0.28\hsize]{#1}}
\newcommand{\igl}[1]{\includegraphics[width=0.9\hsize]{#1}}
\newcommand{\igs}[1]{\includegraphics[width=0.18\hsize]{#1}}
\newcommand{\cSap}[1]{{\cS_{\rm ap}({#1})}}
\def\sI{{\mathsf I}}
\def\sJ{{\mathsf J}}
\def\sM{{\mathsf M}}
\def\sS{{\mathsf S}}
\def\eps{{\varepsilon}}
\def\N{{\mathbb N}}
\def\C{{\mathbb C}}
\def\Z{{\mathbb Z}}
\def\R{{\mathbb R}}
\def\Q{{\mathbb Q}}
\def\cD{{\mathcal{D}}}
\def\cM{{\mathcal{M}}}
\def\cS{{\mathcal{S}}}
\def\cH{{\mathcal{H}}}
\def\cM{{\mathcal{M}}}
\def\cR{{\mathcal{R}}}
\def\cC{{\mathcal{C}}}
\def\cT{{\mathcal{T}}}
\def\bp{{\mathbf p}}
\def\bq{{\mathbf q}}
\def\bP{{\mathbf P}}
\def\bG{{\mathbf G}}
\def\Gal{{\mathrm{Gal}}}
\def\et{\text{\'et}}
\def\ab{\mathrm{ab}}
\def\proP{{\text{pro-}p}}
\def\padic{{p\mathchar`-\mathrm{adic}}}
\def\la{\langle}
\def\ra{\rangle}
\def\scM{\mathscr{M}}
\def\lala{\la\!\la}
\def\rara{\ra\!\ra}
\def\ttx{{\mathtt{x}}}
\def\tty{{\mathtt{y}}}
\def\ttz{{\mathtt{z}}}
\def\bkappa{{\boldsymbol \kappa}}
\def\scLi{{\mathscr{L}i}}
\def\sLL{{\mathsf{L}}}
\def\GL{\mathrm{GL}}
\def\Coker{\mathrm{Coker}}
\def\area{\mathrm{area}}
\def\Ker{\mathrm{Ker}}
\def\CHplus{\underset{\mathsf{CH}}{\oplus}}
\def\check{{\clubsuit}}
\def\kaitobox#1#2#3{\fbox{\rule[#1]{0pt}{#2}\hspace{#3}}\ }
\def\vru{\,\vrule\,}
\newcommand*{\longhookrightarrow}{\ensuremath{\lhook\joinrel\relbar\joinrel\rightarrow}}
\newcommand{\hooklongrightarrow}{\lhook\joinrel\longrightarrow}
\def\nyoroto{{\rightsquigarrow}}
\newcommand{\pathto}[3]{#1\overset{#2}{\dashto} #3}
\newcommand{\pathtoD}[3]{#1\overset{#2}{-\dashto} #3}
\def\dashto{{\,\!\dasharrow\!\,}}
\def\ovec#1{\overrightarrow{#1}}
\def\isom{\,{\overset \sim \to  }\,}
\def\GT{{\widehat{GT}}}
\def\bfeta{{\boldsymbol \eta}}
\def\brho{{\boldsymbol \rho}}
\def\bomega{{\boldsymbol \omega}}
\def\sha{\scalebox{0.6}[0.8]{\rotatebox[origin=c]{-90}{$\exists$}}}
\def\upin{\scalebox{1.0}[1.0]{\rotatebox[origin=c]{90}{$\in$}}}
\def\downin{\scalebox{1.0}[1.0]{\rotatebox[origin=c]{-90}{$\in$}}}
\def\torusA{{\epsfxsize=0.7truecm\epsfbox{torus1.eps}}}
\def\torusB{{\epsfxsize=0.5truecm\epsfbox{torus2.eps}}}
\def\Conf{{\mathrm{Conf}}}
\newcommand{\tvect}[3]{%
   \ensuremath{\Bigl(\negthinspace\begin{smallmatrix}#1\\#2\\#3\end{smallmatrix}\Bigr)}}
\newcommand{\bvect}[2]{%
   \ensuremath{\Bigl(\negthinspace\begin{smallmatrix}#1\\#2\end{smallmatrix}\Bigr)}}
\newcommand{\bmatx}[4]{%
   \ensuremath{\Bigl(\negthinspace\begin{smallmatrix}#1&#2\\#3&#4\end{smallmatrix}\Bigr)}}

\def\bbS{{\mathbb S}}	%
\def\bbSap{{\mathbb S_{\mathrm{ap}}}}	%
\def\bbT{{\mathbb T}}
\def\Proj{{\mathbb P}}
\def\RP#1{{\R{\mathbb P}^{#1}}}
\def\QP#1{{\Q{\mathbb P}^{#1}}}
%
\def\diag{{\mathrm{diag}}}
\def\sK{{\mathsf K}}	%
\def\sN{{\mathsf N}}	%
\def\Mlbl{{\ttw\ttx/\tty\ttz}}
\def\Mlblc{{\ttw\ttx/\ttz\tty}}
\def\tta{{\mathtt{a}}}
\def\ttb{{\mathtt{b}}}
\def\ttc{{\mathtt{c}}}
\def\ttw{{\mathtt{w}}}
\def\ttx{{\mathtt{x}}}
\def\tty{{\mathtt{y}}}
\def\ttz{{\mathtt{z}}}
\def\SetMlbl{\{\tta\tta/\tta\ttb$, $\tta\ttb/\tta\ttb$, $\tta\ttc/\tta\ttb\}}

\renewcommand{\labelenumi}{\textrm{\theenumi.}}


\title{On generalized median triangles \\ and tracing orbits}
\author{Hiroaki Nakamura}
\address{
Department of Mathematics, 
Graduate School of Science, 
Osaka University, 
Toyonaka, Osaka 560-0043, Japan}
\email{nakamura@math.sci.osaka-u.ac.jp}

\author{Hiroyuki Ogawa}
\address{
Department of Mathematics, 
Graduate School of Science, 
Osaka University, 
Toyonaka, Osaka 560-0043, Japan}
\email{ogawa@math.sci.osaka-u.ac.jp}


\begin{abstract}
We study generalization of median triangles on the plane 
with two complex parameters.
By specialization of the parameters, we produce periodical motion 
of a triangle whose vertices trace each other on a common closed
orbit.
\end{abstract}

\maketitle

\markboth{H.Nakamura, H.Ogawa}
{On generalized median triangles and tracing orbits}

\vspace*{-5mm}

\tableofcontents
\footnote[0]{This is a pre-print of an article published in Results in Mathematics. The final authenticated version is available online at: https://doi.org/s00025-020-01268-3}


\vspace{-2\baselineskip}
\section{Introduction}

Given a triangle $\Delta=\Delta ABC$ on a plane, one forms its 
{\it medial} (or {\it midpoint}) {\it triangle}
$\cS(\Delta)=\Delta A'B'C'$ which, by definition, is a triangle 
obtained by joining the midpoints 
$A',B' ,C'$ of the sides $BC,CA,AB$ respectively.
The {\it median triangle}
$\cM(\Delta)=\Delta A''B''C''$ of $\Delta=\Delta ABC$ is
a triangle whose three sides are parallel to the three medians 
$AA',BB',CC'$ of $\Delta$. To position $\cM(\Delta)$, 
it is convenient to impose extra condition
that  $\cM(\Delta)$ shares its centroid with 
$\Delta$ as well as with $\cS(\Delta)$.
To fix labels of vertices of $\cM(\Delta)$, one can
set, for example, $\ovec{AA'}=\ovec{A''B''}$,
$\ovec{BB'}=\ovec{B''C''}$,
$\ovec{CC'}=\ovec{C''A''}$.
 
Arithmetic interest on median triangles can be traced back to 
Euler who found a smallest triangle made of
three integer sides and three integer medians: 
there exists
$\Delta ABC$ with
$\overline{AB}=136$, $\overline{BC}=174$, 
$\overline{CA}=170$, $\overline{AA'}=127$, 
$\overline{BB'}=131$ and
$\overline{CC'}=158$
 (cf.\,\cite{B89}).
In recent years, geometrical constructions of {\it nested triangles}
in more general senses call attentions of researchers
(e.g.,\,\cite{BC12},\cite{Nic13}). In particular,   
M.Hajja \cite{H09}
studied a generalization of the above constructions 
$\cS(\Delta)$ and $\cM(\Delta)$ by introducing
a real parameter $s\in \R$ 
to replace the midpoints of the sides 
by more general $(s:1-s)$-division points. 
Recently in \cite{NO2}, the former construction for $\cS(\Delta)$
was generalized so as to have two complex parameters 
$\Delta\mapsto \cS_{p,q}(\Delta)$
($p,q\in\C$, $pq\ne 1$). 

The primary aim of the first part of this paper is,
following the line of \cite{NO2}, 
to extend the procedure for $\cM(\Delta)$
to a collection of operations 
of the forms
$\Delta\mapsto \cM_{p,q}^{\ttw\ttx/\tty\ttz}(\Delta)$
so that the sides of $\cM_{p,q}^{\ttw\ttx/\tty\ttz}(\Delta)$
are given by vectors joining vertices of $\Delta$ and 
of $\cS_{p,q}(\Delta)$ in 18-fold ways of label correspondences
(See Definition \ref{def-pq-median} below).
After studying mutual relations of the 18-fold ways, we will find that
only three ways among them are essential.
Then, applying the finite Fourier transforms of triangles, we 
obtain operators $\cS[\eta,\eta']$ and $\cM^{\ttw\ttx/\tty\ttz}[\eta,\eta']$
which behave smoothly with the parameter $(\eta,\eta')$ running over 
the full space $\C^2$ (the former was already closely studied in \cite{NO2}).

In the second part of the present paper, we will
study `dancing' of 
triangles  $\cS[\eta,\eta'](\varDelta)$ and
$\cM^{\ttw\ttx/\tty\ttz}[\eta,\eta'](\varDelta)$
along with periodical parameters 
$(\eta(t),\eta'(t))\in\C^2$ ($t\in\R/\Z$).
In particular, we search conditions under which the three vertices of
a triangle trace one after the other in motion along a single common orbit.
Basic examples including ``choreographic three bodies dancing
on a figure eight'' will also be illustrated.

The organization of this paper reads as follows.
In \S 2, we formulate the generalized median operator $\cM_{p,q}^{\ttw\ttx/\tty\ttz}$
on triangles with two complex parameters $p,q$ ($pq\ne 1$) and
with labels $\ttw,\ttx,\tty,\ttz\in\Z/3\Z$ ($\tty\ne\ttz$), 
and illustrate their geometric features on triangles. 
In \S 3, we present how the finite Fourier transformation of triangles 
improves defects of the original parameters $(p,q)$ so as to 
introduce $\cM^{\ttw\ttx/\tty\ttz}[\eta,\eta']$ with
a new parameter system $(\eta,\eta')\in\C^2$.
In particular, $\cM^{\ttw\ttx/\tty\ttz}[\eta,\eta']$ turns out to
be expressed as the generalized cevian operator 
$\cS[\eta_0,\eta_1]$ studied in \cite{NO2}
with suitable change of variables $(\eta,\eta')\to(\eta_0,\eta_1)$
(Corollary \ref{MStran}).
In \S 4, we provide a set of symmetric identities among 
those operators $\cM^{\ttw\ttx/\tty\ttz}[\eta,\eta']$ 
with variations of labels $ \ttw\ttx/\tty\ttz$ and of parameters
$(\eta,\eta')$, and conclude the prescribed primary goal of the first part
of this paper.
A short section \S 5 is then inserted to introduce the space of
triangle shapes (moduli space of similarity classes) from the
viewpoint of finite Fourier transformation and Hajja's shape function. 
We also discuss relationship between Hajja's median operator
$\cH_s$ and a binary Ceva operator $\cC_s$ of Griffiths, B\'enyi-\'Curgus type
from our viewpoint in complex parameter $s\in\C$.
The final section \S 6 is devoted to studying tracing orbits of three bodies
and present their primary characterization in the form
$\cS[\eta(t),\eta'(t)](\Delta_0)$ with certain continuous periodic 
functions $\eta(t),\eta'(t):\R/\Z\to \C$. 
We illustrate some examples of area preserving triangle motions and 
of figure eight orbits. The latter example will be generalized to 3-braiding
motions on Lissajous curves in a separate article \cite{KNO20}.

\section{Generalized median operators}
%

Throughout this paper, we use the notations: 
$i:=\sqrt{-1}$, $\rho:=e^{2 \pi i/6}$, $\omega:=e^{2 \pi i/3}$.

We consider any triangle lies on the complex plane $\C$
and identify it with the multiset of vertices $\{a_0,a_1,a_2\}$
on $\C$.
It is useful to say that a vector $\Delta=(a_0,a_1,a_2)\in\C^3$
is a {\it triangle triple} 
representing the triangle $\{a_0,a_1,a_2\}$.
A triangle triple $\Delta=(a_0,a_1,a_2)\in\C^3$ will sometimes be 
written as
$\Delta=(a_\ttx)_{\ttx\in\Z/3\Z}$ after
the index set $\{0,1,2\}$ for coordinates being
naturally identified 
with $\Z/3\Z$, the ring of integers modulo 3.

In \cite{NO2}, for $p,q\in\C$ with $pq\ne 1$, 
we introduced an operation $\cS_{p,q}$ on the triangle triples 
defined by
\begin{equation} \label{eq1.1}
\cS_{p,q}(a_0,a_1,a_2)=(a_0',a_1',a_2'):
\begin{cases}
a_0'=\hspace{-4mm}& \ \alpha_{p,q}\, a_0 + \beta_{p,q}\, a_1 +\gamma_{p,q} \, a_2 ;\\
a_1'=\hspace{-4mm}& \ \alpha_{p,q}\, a_1 + \beta_{p,q}\, a_2 +\gamma_{p,q} \, a_0 ;\\
a_2'=\hspace{-4mm}& \ \alpha_{p,q}\, a_2 + \beta_{p,q}\, a_0 +\gamma_{p,q} \, a_1,
\end{cases}
\end{equation}
where,  
$$
\alpha_{p,q}=\frac{p(1-q)}{1-pq},
\  
\beta_{p,q}=\frac{q(1-p)}{1-pq},
\
\gamma_{p,q}=\frac{(1-p)(1-q)}{1-pq}.
$$
When $p,q$ are real numbers, $\cS_{p,q}(\Delta)$ can be obtained 
from intersection points of certain two cevian triples of $\Delta$ 
as introduced in \cite{NO03}. For convenience, 
we shall call $\cS_{p,q}$ a {\it generalized cevian operator} 
on triangles also for complex parameters $p,q$.
Since $\alpha_{p,q}+\beta_{p,q}+\gamma_{p,q}=1$, 
it is easy to see that the centroids of
$\Delta=(a_0,a_1,a_2)$ and of $\Delta':=\cS_{p,q}(\Delta)=(a_0',a_1',a_2')$
coincide and that 
\begin{equation} \label{condition0}
\sum_{k\in\Z/3\Z}
\ovec{a_{\ttw+k}a_{\ttx+k}'}=\mathbf{0}
\end{equation}
for any choice of $\ttw,\ttx\in\Z/3\Z$.
This determines, 
for each $(\tty,\ttz)\in(\Z/3\Z)^2$ with $\tty\ne\ttz$, 
a unique triangle triple 
$\Delta''=(a_0'',a_1'',a_2'')$ by the conditions:
\begin{align}
&\text{ $\Delta''$ shares the centroid with $\Delta,\Delta'$, 
}  
\label{condition1}
\\
&\text{ in other words, the three triangles $\Delta,\Delta'$ and $\Delta''$ are concentroid;} 
\nonumber
\\
&\ovec{a_{\ttw+k}a_{\ttx+k}'}=\ovec{a_{\tty+k}''a_{\ttz+k}''}
\quad  (k\in\Z/3\Z).
\label{condition2} 
\end{align}

\begin{Definition}[$(p,q)$-median triangle] \label{def-pq-median}
Let $p,q\in\C$ with $pq\ne 1$, and $\ttw,\ttx,\tty,\ttz\in\Z/3\Z$ with
$\tty\ne\ttz$. 
Given a triangle triple $\Delta=(a_0,a_1,a_2)$ with 
$\Delta'=\cS_{p,q}(\Delta)=(a_0',a_1',a_2')$, 
we define the triangle triple
$$
\cM_{p,q}^{\ttw\ttx/\tty\ttz}(\Delta):=\Delta'',
$$
where $\Delta''=(a_0'',a_1'',a_2'')$ is determined by the condition
(\ref{condition1})-(\ref{condition2}).
We shall call $\cM^\Mlbl_{p,q}$ a {\it generalized median operator} 
on triangles.
\end{Definition}

Before focusing on specific examples, let us here
illustrate actions of $\cM^\Mlbl_{p,q}$ in a generic sample case:
Let $\Delta=(a_0,a_1,a_2)$ be a triangle $(0,1,\frac{7+8i}{10})$
and let complex parameter $(p,q)$ be set as $(\frac45,\frac{2+4i}{3})$.
Then, the generalized cevian operator $\cS_{p,q}$ maps $\Delta$ to
$\Delta'=(a_0',a_1',a_2')=
(\frac{93}{3050}+\frac{542}{1525}i, \frac{201}{305}-\frac{46}{305}i,
\frac{1541}{1525}+\frac{908}{1525}i)$.
One can form a triangle $\Delta''=(a_0'',a_1'',a_2'')$ formed by the sides
parallel to the three vectors $\ovec{a_0a_0'}$, $\ovec{a_1a_1'}$, $\ovec{a_2a_2'}$.
Here arise six-fold ways to label the vertices $a_0'',a_1'',a_2''$ depending on
choices of pairs $(\tty,\ttz)$ with $\tty,\ttz\in\{0,1,2\}$, $\tty\ne\ttz$
so that $\ovec{a_0 a_0'}=\ovec{a_\tty'' a_\ttz''}$.
Every such a choice yields the generalized
median triangle $\cM_{p,q}^{00/\tty\ttz}(\Delta)$.
The following picture (Figure \ref{Fig1}) 
illustrates
$\cM_{p,q}^{00/01}(\Delta)$,
one of those six choices, such that
$\ovec{a_0a_0'}=\ovec{a_0''a_1''}$.
We also note that $\cM_{p,q}^{00/01}=\cM_{p,q}^{11/12}=\cM_{p,q}^{22/20}$
by definition.

\begin{figure}[htbp]
\begin{center}
\begin{tabular}{cc}
\begin{minipage}{0.5\hsize}
\begin{center}
\includegraphics[scale=0.36, width=0.9\hsize]{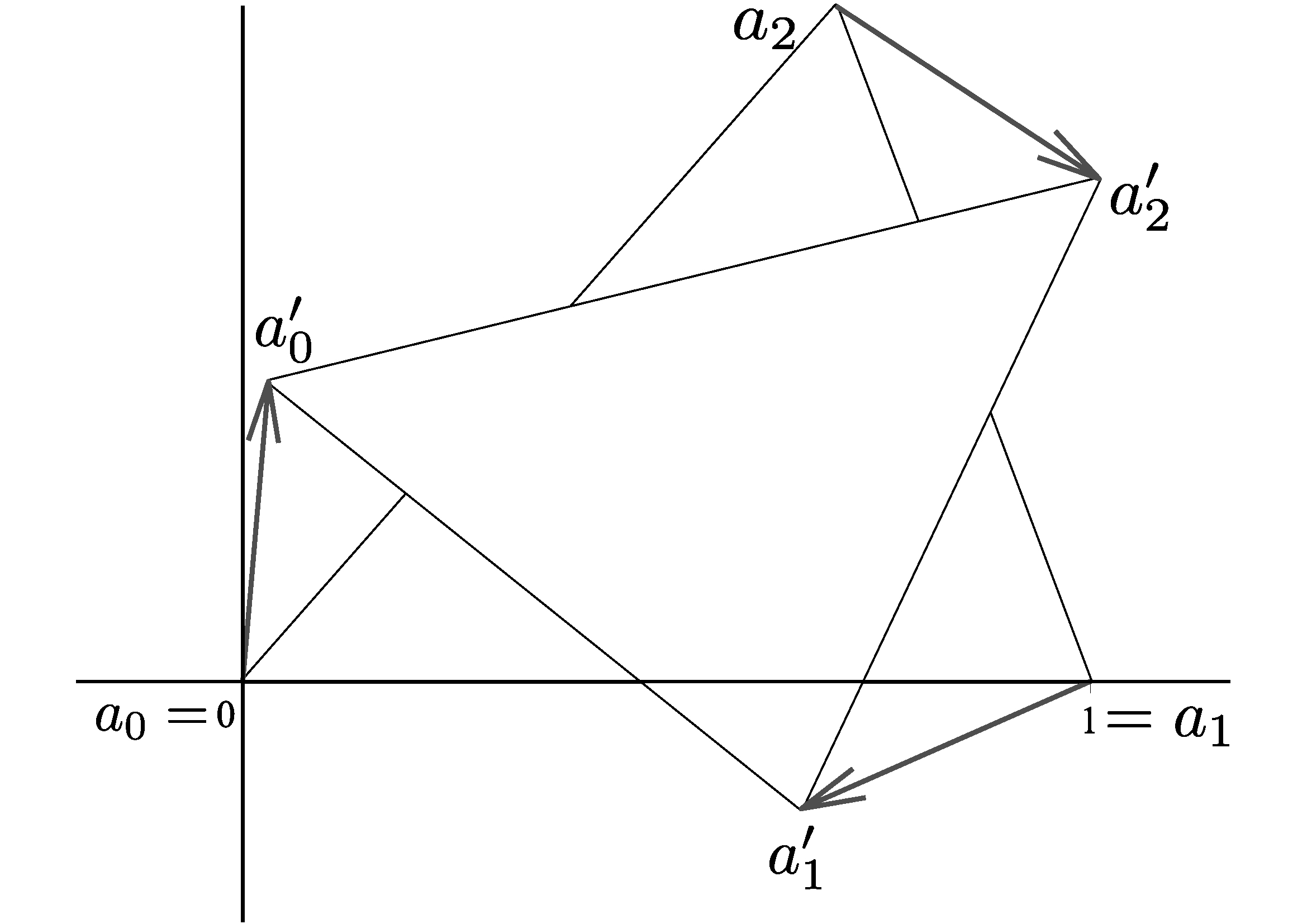}
\caption*{$\Delta=(0,1,\frac{7+8i}{10})$ and $\Delta'=\cS_{\frac45,\frac{2+4i}{3}}(\Delta)$}
\end{center}
\end{minipage} 
&
\begin{minipage}[c]{0.5\hsize}
\centering
\includegraphics[scale=0.36, width=0.9\hsize]{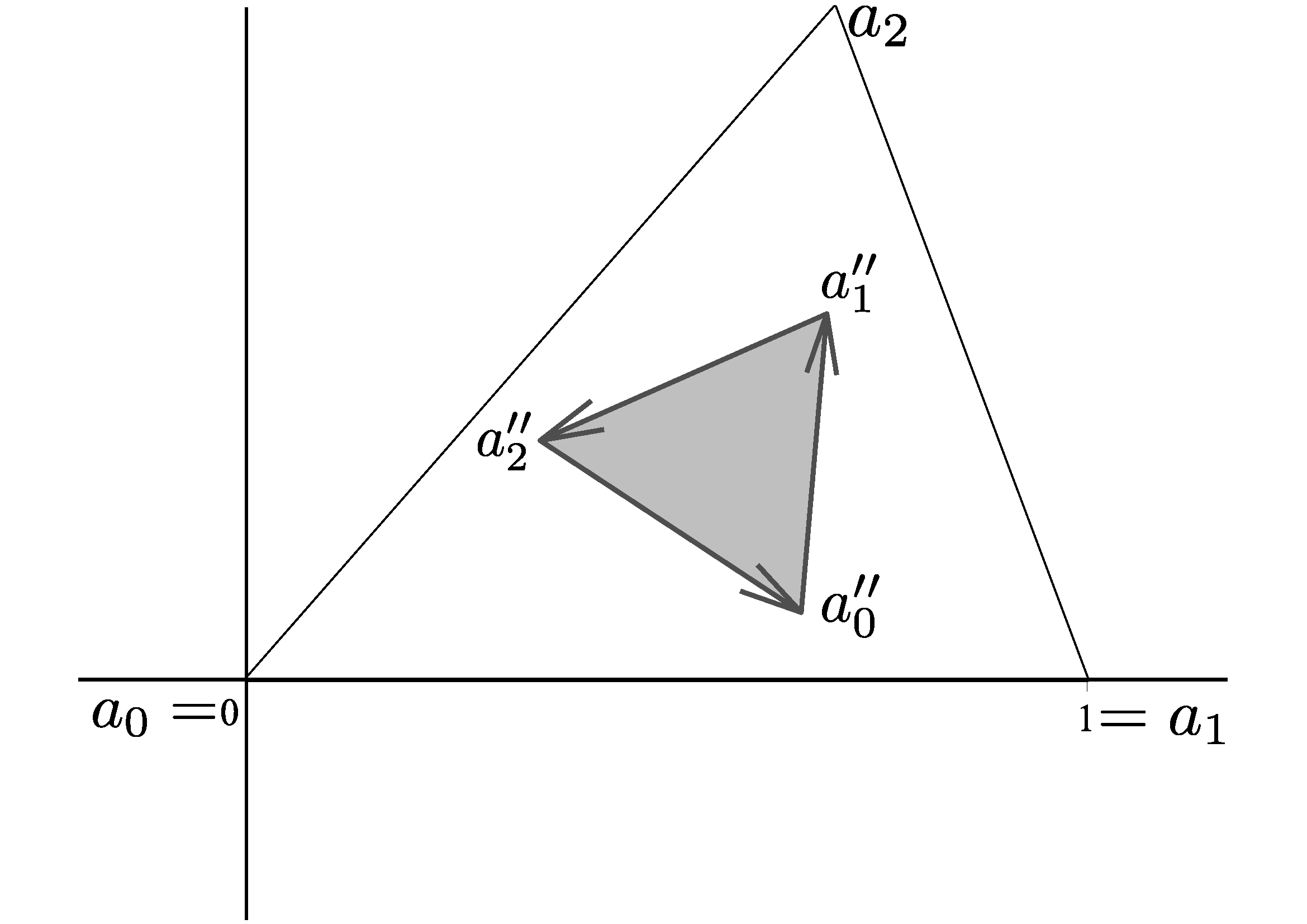}
\caption*{$\Delta$ and $\Delta''=\cM^{00/01}_{\frac45,\frac{2+4i}{3}}(\Delta)$}
\end{minipage} \\
\end {tabular}
\caption{Illustration of $\mathcal{S}_{p,q}$ and $\mathcal{M}_{p,q}^{00/01}$}\label{Fig1}
\end{center}
\end{figure}

\noindent
There is another set of six-fold ways to form 
$\Delta''=(a_0'',a_1'',a_2'')$ whose sides are taken to be 
parallel to $\ovec{a_0a_1'}$, $\ovec{a_1a_2'}$, $\ovec{a_2a_0'}$
in total.
The following picture (Figure \ref{Fig2})
shows one of those cases 
$\cM^{01/01}_{p,q}(\Delta)$
where $a_0'',a_1'',a_2''$ are labeled to satisfy
$\ovec{a_0a_1'}=\ovec{a_0''a_1''}$.
We also note that $\cM_{p,q}^{01/01}=\cM_{p,q}^{12/12}=\cM_{p,q}^{20/20}$
by definition.

\begin{figure}[htbp]
\begin{center}
\begin{tabular}{cc}
\begin{minipage}[c]{0.48\hsize}
\centering
\includegraphics[scale=0.36, width=0.8\hsize]{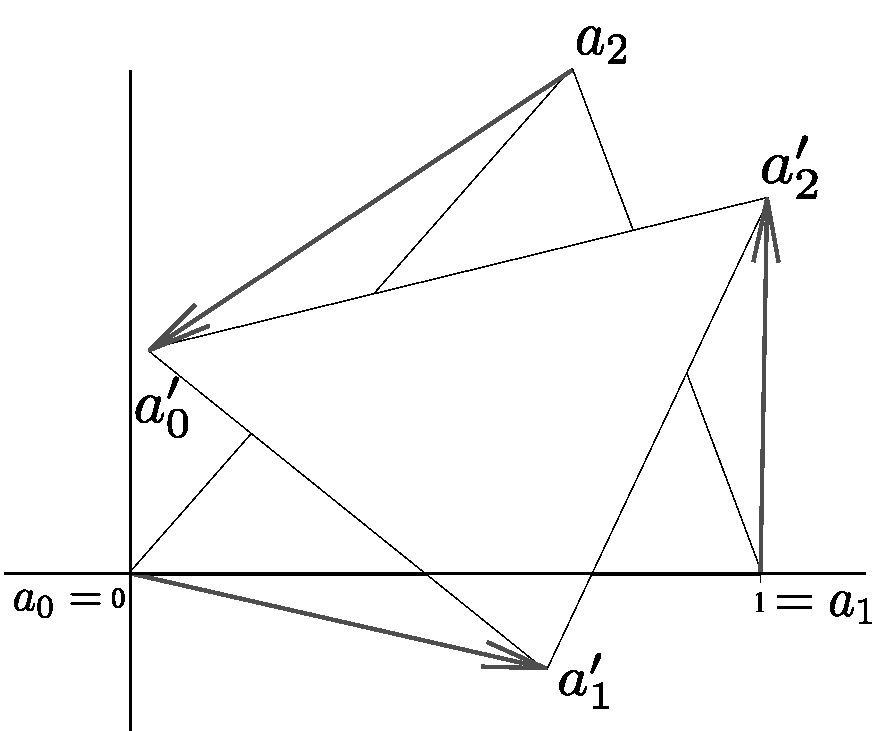}
\caption*{$\Delta=(0,1,\frac{7+8i}{10})$ and $\Delta'=\cS_{\frac45,\frac{2+4i}{3}}(\Delta)$}
\end{minipage}  &
\begin{minipage}[c]{0.48\hsize}
\centering
\includegraphics[scale=0.36, width=0.8\hsize]{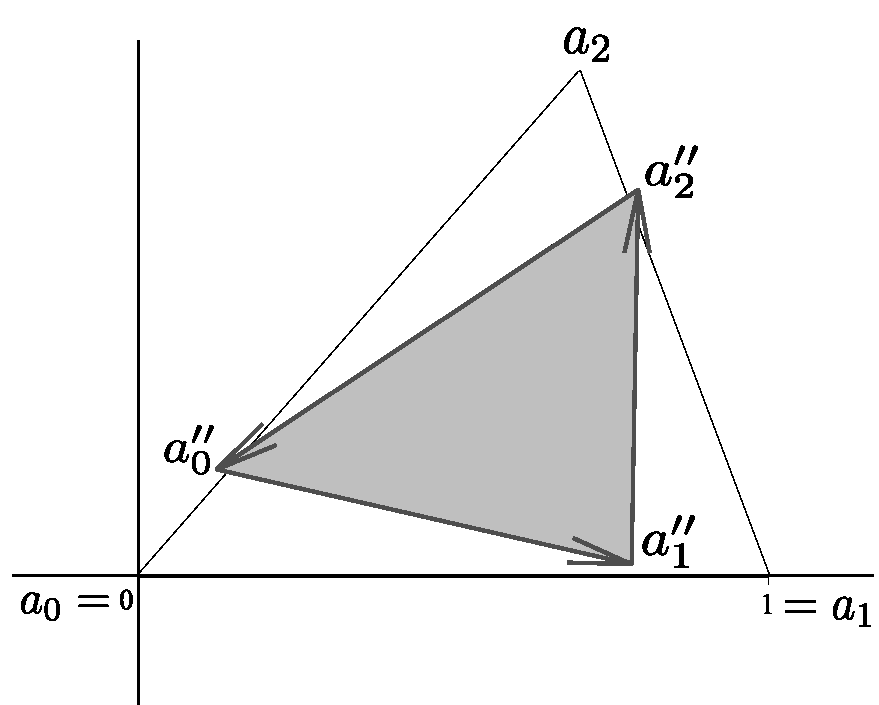}
\caption*{$\Delta$ and $\Delta''=\cM^{01/01}_{\frac45,\frac{2+4i}{3}}(\Delta)$}
\end{minipage} \\
\end {tabular}
\caption{Illustration of $\mathcal{S}_{p,q}$ and $\mathcal{M}_{p,q}^{01/01}$ }\label{Fig2}
\end{center}
\end{figure}

\noindent
It remains to take $\Delta''=(a_0'',a_1'',a_2'')$ formed by three
sides parallel to  $\ovec{a_0a_2'}$, $\ovec{a_1a_0'}$, $\ovec{a_2a_1'}$.
Again we have six-fold ways to label the vertices of $\Delta''$
subject to $\ovec{a_0a_2'}=\ovec{a_{\tty}''a_{\ttz}''}$
($\tty,\ttz\in\{0,1,2\}$, $\tty\ne\ttz$).
The following picture (Figure \ref{Fig3})
illustrates the case $\tty=0,\ttz=1$.
We also note that $\cM_{p,q}^{02/01}=\cM_{p,q}^{10/12}=\cM_{p,q}^{21/20}$
by definition.

\begin{figure}[htbp]
\begin{center}
\begin{tabular}{cc}
\begin{minipage}[c]{0.48\hsize}
\centering
\includegraphics[scale=0.36, width=0.8\hsize]{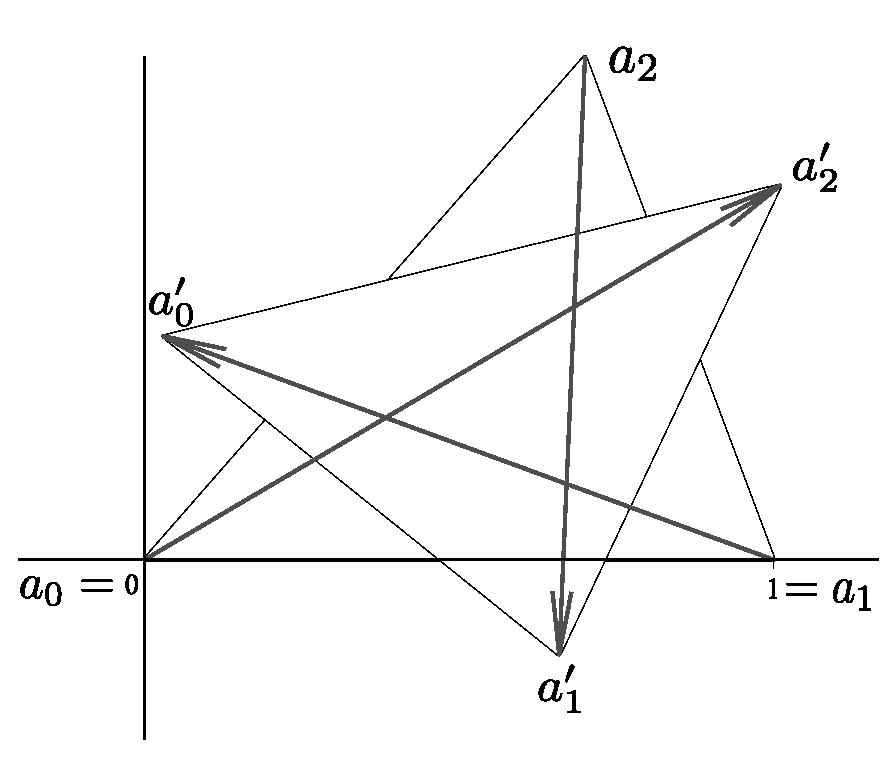}
\caption*{$\Delta=(0,1,\frac{7+8i}{10})$ and $\Delta'=\cS_{\frac45,\frac{2+4i}{3}}(\Delta)$}
\label{tb:table2}
\end{minipage} &
\begin{minipage}[c]{0.48\hsize}
\centering
\includegraphics[scale=0.36, width=0.8\hsize]{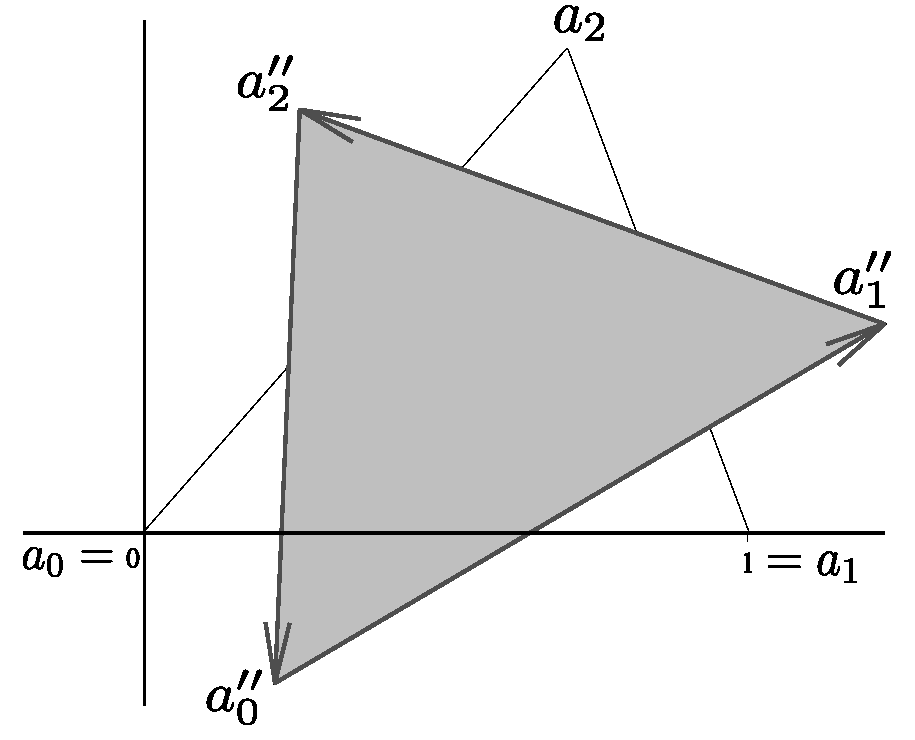}
\caption*{$\Delta$ and $\Delta''=\cM^{02/01}_{\frac45,\frac{2+4i}{3}}(\Delta)$}
\end{minipage}
\end {tabular}
\caption{Illustration of $\mathcal{S}_{p,q}$ and $\mathcal{M}_{p,q}^{02/01}$ }\label{Fig3}
\end{center}
\end{figure}

\noindent
As shown in the above description, we generally have 18($=3\times 6$)-fold 
ways to define $\Delta''=\cM_{p,q}^{\ttw\ttx/\tty\ttz}(\Delta)$ whose sides 
are composed of three disjoint bridges between the vertices of
$\Delta$ and of $\cS_{p,q}(\Delta)$. 
In \S 4, we will discuss precise relations among those 18-fold ways
at the operator level.

In the next two examples, we focus on some specific cases which connect 
$\cM_{p,q}^{\ttw\ttx/\tty\ttz}(\Delta)$
to its historical origins.

\begin{Example}[Prototype]\label{Prototype}
Let $\Delta ABC$ be a triangle represented by a
triple $\Delta=(a,b,c)\in\C^3$.
Let us illustrate the classical case in Introduction in our terminology:
As noted in \cite[Example 1.2]{NO03}, the midpoint triangle 
$\cS(\Delta)=\Delta A'B'C'$ is given by 
$\cS_{0,\frac{1}{2}}(\Delta)$.
The median triangle $\cM(\Delta)=\Delta A''B''C''$ labeled by the condition
$\ovec{AA'}=\ovec{A''B''}$,
$\ovec{BB'}=\ovec{B''C''}$,
$\ovec{CC'}=\ovec{C''A''}$
is then given by $\cM_{0,\frac{1}{2}}^{00/01}(\Delta)$.
\end{Example}

\begin{Example}\label{Hajja}
In \cite{H09}, M.Hajja discusses three types of triangles called 
the $s$-medial, the $s$-Routh, and the $s$-median triangles with a real parameter $s\in \R$. 
The $(p,q)$-median triangle introduced above generalizes Hajja's $s$-median triangle.  
Start with a triangle $\Delta ABC$ represented by a 
{\it positive} triangle triple $\Delta=(a,b,c)$ satisfying $\mathrm{Im}(\frac{a-b}{c-b})>0$.
Form first $\Delta'=(a',b',c')$ to be $\cS_{0,1-s}(\Delta)$ 
(called the $s$-medial triangle of $\Delta$),  
the triangle whose vertices are $(s:1-s)$-division points of the edges of $\Delta$. 
The $s$-median triangle of $\Delta$, written $\cH_s(\Delta)$ is, 
by definition, a triangle $\{a'',b'',c''\}$ such that 
$\ovec{aa'}=\ovec{b''c''}$, $\ovec{bb'}=\ovec{c''a''}$, and 
$\ovec{cc'}=\ovec{a''b''}$.
Without loss of generality, we may assume $\cH_s(\Delta)$ and $\Delta$ are concentroid, i.e. $a+b+c=a''+b''+c''$ so that $\cH_s(\Delta)$ is uniquely determined from $\Delta$. 
In our above definition, we find $\cH_s(\Delta)$ to be
$\cM^{00/12}_{0,1-s}(\Delta)$. 
\end{Example}

\section{Fourier parameters}
The collection of operators 
$\bbS':=\{\cS_{p,q}\mid (p,q)\in\C^2, pq\ne 1\}$ is incomplete 
in the sense that 
the composition $\cS_{p_1,q_1}\cS_{p_2,q_2}$ may not always be
of the form of an $\cS_{p,q}\in\bbS'$.
The lesson found in our previous work \cite{NO2} to remedy this defect 
is to introduce the Fourier transforms $\Psi(\Delta)$ for 
triangles $\Delta=(a,b,c)$ by
\begin{equation}
\label{fourier}
\Psi(\Delta)=\tvect{\psi_0(\Delta)}{\psi_1(\Delta)}{\psi_2(\Delta)}
=\frac13
\tvect{a+b+c}{a+b\omega^2+c\omega}{a+b\omega+c\omega^2}
\end{equation}
and to replace the parameter $(p,q)\in\C^2$ ($pq\ne 1$)
by a new parameter $(\eta,\eta')\in \C^2$ defined by 
\begin{equation} \label{etafrompq}
\eta:=\frac{p-q}{1-pq}+\frac{(p-1)(2q-1)}{1-pq}\omega, \quad 
\eta':=\frac{p-q}{1-pq}+\frac{(p-1)(2q-1)}{1-pq}\omega^2.
\end{equation}
Indeed, with these parameters, the operator 
$\cS_{p,q}$ is diagonalized as
mapping $\Delta$ to $\Delta'$ in the form 
\begin{equation} \label{EtaTimesPsi}
\psi_0(\Delta')=\psi_0(\Delta),\quad
\psi_1(\Delta')=\eta' \cdot\psi_1(\Delta),\quad
\psi_2(\Delta')=\eta \cdot\psi_2(\Delta).
\end{equation}
It turns out that the collection 
$\bbS':=\{\cS_{p,q}\mid (p,q)\in\C^2, pq\ne 1\}$
extends to a more complete family
\begin{equation}
\bbS:=\{\cS[\eta,\eta']\mid (\eta,\eta')\in\C^2\}
\end{equation}
by identifying 
$\cS_{p,q}=\cS[\eta,\eta']$ for $pq\ne 1$
so that the
composition law 
$\cS[\eta_1,\eta_1']\cS[\eta_2,\eta_2']=\cS[\eta_1\eta_2,\eta_1'\eta_2']$
provides a natural multiplicative monoid structure on $\bbS$.

Now, regarding triangle triples as column vectors in $\C^3$, 
we easily see that the operations 
$\cS_{p,q}$ and $\cS[\eta,\eta']$
naturally determine linear transformations 
(3 by 3 matrices in $M_3(\C)$) 
acting on $\C^3$ on the left. 
Below, we shall identify those operators as their matrix representatives
in $M_3(\C)$.
Let 
$$
 \sI:=\begin{pmatrix} 1 & 0 & 0 \\ 0 & 1 & 0 \\ 0 & 0 & 1 \end{pmatrix}, \quad
 \sJ:=\begin{pmatrix} 0 & 1 & 0 \\ 0 & 0 & 1 \\ 1 & 0 & 0  \end{pmatrix}, \quad
 W:=\begin{pmatrix} 1 & 1 & 1 \\ 1 &\omega & \omega^2 \\1&\omega^2&\omega
\end{pmatrix}.
$$
Note that the above Fourier transform (\ref{fourier}) may be
written in the matrix multiplication form: 
$\Psi\bigl(\tvect{a}{b}{c}\bigr)=W^{-1}\tvect{a}{b}{c} $.
The following proposition summarizes basic properties for 
$\cS[\eta,\eta']\in\bbS$:

\begin{Proposition}[\cite{NO2}] \label{Spq:eqn}
Notations being as above, we have:
\begin{enumerate}[label=(\roman*),font=\upshape]
\item[(i)] $\bbS=\{\alpha\sI+\beta\sJ+\gamma\sJ^2\mid \alpha+\beta+\gamma=1\}\subset M_3(\C)$.
\item[(ii)]
$\cS[\eta,\eta']=W\cdot \diag(1,\eta',\eta)\cdot W^{-1}$ \\
$\qquad\quad
=\frac{1}{3}(1+\eta+\eta')\sI
+\frac{1}{3}(1+\eta\omega+\eta'\omega^2)\sJ
+\frac{1}{3}(1+\eta\omega^2+\eta'\omega)\sJ^2
$ $\ (\eta,\eta'\in\C)$.
\end{enumerate}
\end{Proposition}

Let us turn to generalized median operators.
We first extend $\cM_{p,q}^{\ttw\ttx/\tty\ttz}$ ($pq\ne 1$) to 
the new parameters $(\eta,\eta')\in \C^2$.
Below, we understand the number $\omega^{\ttx}$ and the matrix $\sJ^{\ttx}$ 
in the obvious sense for each $\ttx\in\Z/3\Z$. 

\begin{Definition}[$(\eta,\eta')$-median triangles] \label{eta-median}
Let $\eta,\eta'\in\C$, and let $\ttw,\ttx,\tty,\ttz\in\Z/3\Z$ with
$\tty\ne\ttz$. 
Given a triangle triple $\Delta=(a_0,a_1,a_2)$ with 
$\Delta'=\cS[\eta,\eta'](\Delta)=(a_0',a_1',a_2')$, 
we define the triangle triple
$$
\cM^{\ttw\ttx/\tty\ttz}[\eta,\eta'](\Delta):=\Delta'',
$$
where $\Delta''=(a_0'',a_1'',a_2'')$ is determined by the condition
(\ref{condition1})-(\ref{condition2}).
\end{Definition}
 
It is not difficult to see that $\cM^{\ttw\ttx/\tty\ttz}[\eta,\eta']\in\bbS$.
In fact, we have the following explicit formula:

\begin{Proposition}\label{Mpq:eqn}
Given $\eta,\eta'\in\C$ and  $\ttw,\ttx,\tty,\ttz\in\Z/3\Z$ with
$\tty\ne\ttz$, we have
$$
(\sJ^{\ttz}-\sJ^{\tty})\cM^{\Mlbl}[\eta,\eta']=\sJ^{\ttx}\cS[\eta,\eta']-\sJ^{\ttw}.
$$
\end{Proposition}

\begin{proof}
Let $\Delta=(a_0,a_1,a_2)$ be a triangle triple, and write 
$\Delta'=\cS[\eta,\eta'](\Delta)=(a_0',a_1',a_2')$
and 
$\Delta''=\cM^{\Mlbl}[\eta,\eta'](\Delta)
=(a_0'',a_1'',a_2'')$.
The assertion essentially amounts to seeing the identity
$$
(\sJ^{\ttz}-\sJ^{\tty})(\Delta'')
=\sJ^{\ttx}(\Delta')-\sJ^{\ttw}(\Delta).
$$
Observe that the 1st component of 
$\sJ^{\ttx}(\Delta')-\sJ^{\ttw}(\Delta)$
is $\ovec{a_\ttw a_\ttx'}$, and that 
the 1st component of 
$\sJ^{\ttz}(\Delta'')-\sJ^{\tty}(\Delta'')$
is $\ovec{a_\tty'' a_\ttz''}$.
They coincide with each other by definition.
Similarly, one can see the coincidence of their 2nd and 3rd components,
as they are the 1st components of the above 
after $\Delta$ replaced by $\sJ\Delta$, $\sJ^{2}\Delta$. 
One can extend the identity also for degenerate triangle triples 
by easy argument of continuity, and hence conclude 
the matrix identity as asserted.
\end{proof}

Although the factor $(\sJ^{\ttz}-\sJ^{\tty})$ in LHS of the above 
Proposition \ref{Mpq:eqn}
is not an invertible matrix, the concentroid condition (\ref{condition1})
determines $\cM^{\Mlbl}[\eta,\eta']$ in $\bbS$
as seen in the following corollary.
In fact, the generalized median operator $\cM^{\Mlbl}[\eta,\eta']$
turns out to be reduced to a generalized cevian operator
$\cS[\eta_0,\eta_1]$ after a simple change of parameters:

\begin{Corollary} \label{MStran}
Notations being as in Proposition \ref{Mpq:eqn}, 
we have
$$
\cM^{\Mlbl}[\eta,\eta']=\cS[\eta_0,\eta_1]
$$
where
$$
\eta_0=\frac{\eta\omega^{-\ttx}-\omega^{-\ttw}}{\omega^{-\ttz}-\omega^{-\tty}}, \quad
\eta_1=\frac{\eta'\omega^{\ttx}-\omega^{\ttw}}{\omega^{\ttz}-\omega^{\tty}}.
$$
\end{Corollary}

\begin{proof}
Let $N=\frac{1}{3}(\sI+\sJ+\sJ^2)$ (i.e., the matrix with all entries $\frac13$)
so that $N(\Delta)=(g,g,g)$ for every triangle $\Delta=(a,b,c)$ with centroid
$g=\frac13 (a+b+c)$.
Since $M=\cM^{\Mlbl}[\eta,\eta']$ preserves centroids of triangles,
we have $NM(\Delta)=N(\Delta)$ for all $\Delta$, hence have
the identity $NM=N$. It follows then from Proposition \ref{Mpq:eqn} that
$(\ast):(\sJ^{\ttz}-\sJ^{\tty}+N)M=\sJ^{\ttx}\cS[\eta,\eta']-\sJ^{\ttw}+N$.
Since the matrix $(\sJ^{\ttz}-\sJ^{\tty}+N)\in\bbS$ is invertible, the identity ($\ast$) 
determines $M$ which itself lies in $\bbS$ by Proposition \ref{Spq:eqn} (ii) and
gives rise to
{\small
$$
\begin{bmatrix} 
{0} &  &  \\
 & {\omega^{\ttz}-\omega^{\tty}} &  \\
  &  & {\omega^{-\ttz}-\omega^{-\tty}}
\end{bmatrix}
\begin{bmatrix} 
{1} &  & \\
 & {\eta_1} &  \\
  &   & {\eta_0}
\end{bmatrix}
=
 \begin{bmatrix} 
{1} &  &  \\
 & {\omega^\ttx} &  \\
  &   & {\omega^{-\ttx}}
\end{bmatrix}
 \begin{bmatrix} 
{1} &  &  \\
 & {\eta'} &  \\
  &   & {\eta}
\end{bmatrix}
-
 \begin{bmatrix} 
{1} &  &  \\
 & {\omega^\ttw} & \\
  &   & {\omega^{-\ttw}}
\end{bmatrix}
$$
}

\noindent
after conjugation by $W$. 
This settles the asserted formula on $(\eta_0, \eta_1)$.
\end{proof}

\section{Reduction of 18-fold ways of $\cM^{\Mlbl}$}

The upper label $\ttw\ttx/\tty\ttz$
for a generalzed median operator 
$\cM^{\ttw\ttx/\tty\ttz}[\eta,\eta'] $ 
is to be given from the collection of 
$(\ttw,\ttx,\tty,\ttz)\in(\Z/3\Z)^4$ with $\tty\ne\ttz$.

Since the condition (\ref{condition2}) is stable under simultaneous
shifts of labels in $\Mlbl$, we have the identity
$\cM^{\ttw\ttx/\tty\ttz}[\eta,\eta'] =\cM^{\ttw+1,\ttx+1/\tty+1,\ttz+1}[\eta,\eta']$
which will be listed below in (\ref{eq2.2}).
As a consequence, there are 18 different ways of labels
up to the shifts in $\Z/3\Z$.
However, there are many other identities which co-relate  
generalized cevian and median operators 
as shown in the following list (\ref{eq2.1})-(\ref{eq2.8}).

\begingroup
\allowdisplaybreaks
\begin{align}
&\cS[\eta,\eta'] =\cS[\eta \omega, \eta'\omega^{-1}] \cdot\sJ 
=\sJ \cdot \cS[\eta \omega, \eta'\omega^{-1}]
\label{eq2.1} 
\\
&\cM^{\ttw\ttx/\tty\ttz}[\eta,\eta'] =\cM^{\ttw+1,\ttx+1/\tty+1,\ttz+1}[\eta,\eta']
\label{eq2.2}
\\
&\cM^{\ttw\ttx/\tty\ttz}[\eta,\eta'] \cdot \sJ = 
\cM^{\ttw+1,\ttx+1/\tty\ttz}[\eta,\eta'] 
\label{eq2.3}
\\
&\cM^{\ttw\ttx/\tty\ttz}[\eta,\eta'] \cdot \sJ^2 
=\cM^{\ttw\ttx/\tty+1,\ttz+1}[\eta,\eta']  
\label{eq2.4}
\\
&\cM^{\ttw\ttx/\tty\ttz}[\eta,\eta'] 
=\cM^{\ttw,\ttx+1/\tty\ttz}[\eta\omega,\eta'\omega^{-1}] =
\cM^{\ttw,\ttx-1/\tty\ttz}[\eta\omega^{-1},\eta'\omega] 
\label{eq2.5}
\\
&\frac{1}{3}\left(
\cM^{\ttw 0/\tty\ttz}[\eta,\eta']+\cM^{\ttw 1/\tty\ttz}[\eta,\eta']+\cM^{\ttw 2/\tty\ttz}[\eta,\eta']
\right)
=
\frac13\sJ^{\ttw+\tty+\ttz}+\frac23\sJ^{\ttw-\tty}
\label{eq2.6}
\\
&\frac{1}{2}\left(
\cM^{\ttw \ttx/\tty\ttz}[\eta,\eta']+\cM^{\ttw \ttx/\ttz\tty}[\eta,\eta']
\right)
=\cS_{\frac12,\frac12}=\frac13(\sI+\sJ+\sJ^2) 
\label{eq2.7}
\\
&\cM^\Mlbl[\omega^{\ttx-\ttw}+\eta,\omega^{\ttw-\ttx}+\eta']
=\sJ^\ttx \cdot \cM^{00/\tty\ttz}[1+\eta,1+\eta'].
\label{eq2.8}
\end{align}
\endgroup

\begin{proof}[Proof of  (\ref{eq2.1})-(\ref{eq2.8}).]
Proposition \ref{Spq:eqn} and Corollary \ref{MStran}
enable one to express $\cS[\eta,\eta']$ and $\cM^{\Mlbl}[\eta,\eta']$
in $\bbS$ as explicit 3 by 3 matrices for every
$(\eta,\eta')\in\C^2$ and for $\ttw\ttx/\tty\ttz$.
Then the proofs of these identities, once discovered,
can be easily verified (say, by using symbolic computer systems). %
\end{proof}

{\bf Interpretation:} 
Let $\Delta$ be a triangle triple and fix a pair of parameters $(\eta,\eta')\in\C^2$.
The relation (\ref{eq2.7}) tells that  
$\cM^{\ttw \ttx/\tty\ttz}[\eta,\eta'](\Delta)$ and $\cM^{\ttw \ttx/\ttz\tty}[\eta,\eta'](\Delta)$
are point-symmetrical about the centroid of $\Delta$. This together with 
(\ref{eq2.2}) implies that 
$\cM^{\ttw \ttx/\tty,\tty+1}[\eta,\eta'](\Delta)$ ($\ttw,\ttx, \tty\in\Z/3\Z$) 
give all possible triangle triples up to point symmetry.
Consider, then, effects of (\ref{eq2.3}) and (\ref{eq2.4}) after remarking
that the action of $\sJ$ on triangle triples changes only labels of vertices. 
(Note also that every matrix in $\bbS$ commutes with $\sJ$.)
{}From this we realize that the three median triangles
$$
\cM^{00/01}[\eta,\eta'](\Delta), \ 
\cM^{01/01}[\eta,\eta'](\Delta), \ 
\cM^{02/01}[\eta,\eta'](\Delta)
$$ 
provide all possibly different triangles
from 18-fold triples $\cM^{\ttw \ttx/\tty\ttz}[\eta,\eta'](\Delta)$ 
in $(\ttw,\ttx,\tty,\ttz)\in(\Z/3\Z)^4$ with $\tty\ne\ttz$
 (up to parallelism, point symmetry and label permutations).
Note also that the last three triangles are also dependent 
by a linear relation (\ref{eq2.6}).

\bigskip
As illustrated in \cite[Remark 3.6]{NO2}, the operations 
$\cS_{p,q}$ have closer geometrical interpretation on triangles
with respect to the original parameters 
$p,q\in\C$ ($pq\ne 1$). 
In the rest of this section, we shall translate the above results
for $\cS[\eta,\eta']$, $\cM^{\Mlbl}[\eta,\eta']$ into the context
of $\cS_{p,q}$, $\cM^{\Mlbl}_{p,q}$.

Noting that the transformation (\ref{etafrompq}) is birational with
\begin{equation} \label{pqfrometa}
p=\frac{1+\eta+\eta'}{2-\omega \eta -\omega^2\,\eta'}, \quad
q=\frac{1+\omega\,\eta+\omega^2\,\eta'}{2-\eta-\eta'}
\end{equation}
(cf. \cite[Prop. 5.13]{NO2}), we translate Corollary \ref{MStran} in
the form 
\begin{equation} \label{MStranpq}
\cM^{\Mlbl}_{p,q}=\cS_{p_1,q_1}
\end{equation}
with $p_1,q_1$ suitable rational functions in $p,q$ and vice versa.
The following table shows some samples chosen from 18 types of labels, 
where 
$$
\Delta''=(a_0'',a_1'',a_2'')=\cM^{\Mlbl}_{p,q}(\Delta)=\cS_{p_1,q_1}(\Delta)
$$ 
for $\Delta=(a_0,a_1,a_2)$ and $\Delta'=(a_0',a_1',a_2')=\cS_{p,q}(\Delta)$:

{\footnotesize
\begin{table}[H]
\caption[$\cM^{\Mlbl}_{p,q}(\Delta)=\cS_{p_1,q_1}(\Delta)$]{$\cM^{\Mlbl}_{p,q}(\Delta)=\cS_{p_1,q_1}(\Delta)$ }
 \label{table1}  
$$
\begin{array}{|c|c|c|c|}\hline
\Mlbl {\rule{0pt}{4ex}}
& \ovec{a_\ttw a_\ttx'}=\ovec{a_\tty'' a_\ttz''} 
& [p_1,q_1]  
& [p,q] 
\\ 
\hline\hline
 {\rule{0pt}{4ex}} 00/01 & \ovec{a_0a_0'}=\ovec{a_0''a_1''}
&    \vmargin{0ex}{1ex}{ \displaystyle\left[{\frac {2\,pq+p-q-2}{4\,pq-p-2\,q-1}},-{\frac {p-2}{1+p}}
\right]  }
&   \vmargin{0ex}{1ex}{ \displaystyle\left[-\frac {q_1-2}{q_1+1},-\frac{p_1-q_1}{(2\,p_1-1)(q_1-1)}\right] }
\\
\hline
 {\rule{0pt}{4ex}} 01/01 & \ovec{a_0a_1'}=\ovec{a_0''a_1''}
&    \vmargin{0ex}{1ex}{\displaystyle\left[{\frac {4\,pq-2\,p-q-1}{2\,pq-p+q-2}},-{\frac {q+1}{q-2}}\right] }
&   \vmargin{0ex}{1ex}{ \displaystyle\left[ \frac{p_1-q_1}{(p_1-2) (q_1-1)} , \frac{2\,q_1-1}{q_1+1} \right] }
\\ 
\hline
 {\rule{0pt}{4ex}} 02/01 & \ovec{a_0a_2'}=\ovec{a_0''a_1''}
&    \vmargin{0ex}{1ex}{ \displaystyle\left[{\frac {p+2\,q-3}{2\,p+q-3}},{\frac {3\,pq-2\,p-q}{3\,pq-p-2\,q}}\right]  }
&   \vmargin{0ex}{1ex}{ \displaystyle\left[ \frac{(p_1-1) (2\,q_1-1)}{(2\,p_1-1) (q_1-1)} , \frac{(p_1-1) (q_1-2)}{(p_1-2) (q_1-1)} \right] }
\\ 
\hline
 {\rule{0pt}{4ex}} 00/12 & \ovec{a_0 a_0'}=\ovec{a_1''a_2''} 
&    \vmargin{0ex}{1ex}{\displaystyle\left[-\frac{(p-2) (q-1)}{p q+2\,p+q-4} , \frac{(2\,p-1) (q-1)}{4\,p q-p-2\,q-1}\right] }
&   \vmargin{0ex}{1ex}{\displaystyle\left[ \frac{3 p_1 q_1-p_1-2\,q_1}{3\,p_1 q_1-2\,p_1-q_1} , \frac{(2\,p_1-1) (q_1-1)}{(p_1-1) (2\,q_1-1)} \right] }
\\ 
\hline
 {\rule{0pt}{4ex}} 00/20 & \ovec{a_0a_0'}=\ovec{a_2''a_0''} 
&    \vmargin{0ex}{1ex}{\displaystyle\left[{\frac {2\,p-1}{p+1}},{\frac {2\,pq+p-q-2}{pq+2\,p+q-4}}\right]  }
&   \vmargin{0ex}{1ex}{\displaystyle\left[ -\frac{p_1+1}{p_1-2} , \frac{(p_1-1) (2\,q_1-1)}{p_1-q_1} \right] } 
\\ 
\hline
\end{array}
$$
\end{table}
}


\begin{Example}\label{Hajja2}
In Example \ref{Hajja}, we identified Hajja's $s$-median operator 
$\cH_s$ with $\cM^{00/12}_{0,1-s}$ for $s\in\R$.
The above formula (\ref{MStranpq}) (cf. Table \ref{table1})
translates it as 
\begin{equation}
\cH_s=\cM^{00/12}_{0,1-s}=\cS_{\frac{2s}{s+3},\frac{s}{2s-3}}.
\end{equation}
The last expression for $s=-3,\frac32$ appears to be singular 
as $\cS_{\infty,\frac13}$, $\cS_{\frac23,\infty}$ respectively,
but these singularities can be removed 
in the language of $(\eta,\eta')$-parameters:
Indeed, by (\ref{etafrompq}) we can interpret 
$\cS_{0,1-s}=\cS[s\omega+(1-s)\omega^2,s\omega^2+(1-s)\omega]$,
hence from Definition \ref{eta-median}, we obtain
$\cM^{00/12}_{0,1-s}=\cM^{00/12}[s\omega+(1-s)\omega^2,s\omega^2+(1-s)\omega]$.
Corollary \ref{MStran} then allows us to compute
\begin{align}
\cH_s&=\cM^{00/12}[s\omega+(1-s)\omega^2,s\omega^2+(1-s)\omega] \\
&=\cS[s+\omega,s+\omega^2]
=\sJ^2+s \left(\frac23\sI-\frac13\sJ-\frac13\sJ^2\right)
\notag
\end{align}
which makes senses on all $s\in\C$.
Finally, formulas (\ref{eq2.1})-(\ref{eq2.5}) 
transform $\cH_s$ 
into various expressions of generalized medians.
For example, for generic complex parameter $s$, one has:
\begin{equation}
\cH_s
=\cM^{00/01}_{\frac{s-2}{s-1}, \frac{s}{s-1}}
=\cM^{01/01}_{\frac{s}{2},\frac{1}{s-1}}
=\cM^{02/01}_{\frac{1}{1-s},\frac{2-s}{2}}. 
\end{equation}
\end{Example}

\begin{Example}[Parameters for $\cM^{\Mlbl}_{p,q}=\cS_{p,q}$]  \label{MequalS}
Let $\Mlbl$ be a given label with $\ttw,\ttx,\tty,\ttz\in\Z/3\Z$, $\tty\ne\ttz$.
By Proposition \ref{Mpq:eqn}, we find that $\cM^{\Mlbl}[\eta,\eta']=\cS[\eta,\eta']$
has a unique solution in the form
$$
\cS[\eta,\eta']=(\sJ^{\ttx}+\sJ^{\tty}-\sJ^{\ttz})^{-1}\sJ^{\ttw} 
=\alpha \sI +\beta \sJ +\gamma \sJ^2
$$
summarized in the following table
$$
\begin{array}{c@{\ \ [}c@{\,,\,\,}c@{\,,\,\,}c@{]\ \ }|c@{\ \ [}c@{\,,\,\,}c@{\,,\,\,}c@{]\ \ }|c@{\ \ [}c@{\,,\,\,}c@{\,,\,\,}c@{]\ \ }}
 Label & \alpha & \beta & \gamma & Label & \alpha & \beta & \gamma  & Label & \alpha & \beta & \gamma  \\ 
 \hline
 00/01 & 4/7 & 2/7 &1/7 & 
 01/01 & 1 & 0 & 0 &  
 02/01 & 1/2 & 1/2& 0  \\
 00/10 &  0& 0 & 1  &  
 01/10 &  1/7 & 2/7& 4/7   &
 02/10 & 0 & 1/2 & 1/2 \\
 00/02 & 4/7 &  1/7  &2/7 & 
 01/02 &  1/2 &  0 &1/2 & 
 02/02 & 1 & 0 & 0 \\
 00/20 & 0 & 1 & 0 & 
 01/20 & 0 & 1/2 & 1/2 & 
 02/20 & 1/7 & 4/7& 2/7  \\
 00/12 & 1/2 &0 &  1/2 & 
 01/12 & 2/7 & 1/7  &  4/7 &
 02/12 & 0 & 0& 1  \\
 00/21 & 1/2 &1/2 & 0 &  
 01/21 & 0 & 1 & 0 & 
 02/21 & 2/7 & 4/7& 1/7  \\
\end{array}
$$
Recall from Proposition \ref{Spq:eqn} (ii) that
the corresponding parameter $(\eta,\eta')$ for each case
is given by
$$
\begin{cases}
&\eta=\alpha+\beta\omega^2+\gamma\omega, \\
&\eta'=\alpha+\beta\omega+\gamma\omega^2.
\end{cases}
$$
Next we search parameters $(p,q)$ with $pq\ne 1$ satisfying
$\cM^{\Mlbl}_{p,q}=\cS_{p,q}$ from the above table.
They are classified into the following three kinds: 
$$
\begin{tabular}{cll}
 (i) & $\cS_{p,q}\in\la\sJ\ra$ ($=\{\sI,\sJ,\sJ^2\}$) ,
& [$01/01$, $00/10$, $02/02$, $00/20$, $02/12$, $01/21$]; \\
 (ii) & $\cS_{p,q}=\frac12(\sJ^i+\sJ^j)$ $(i\ne j)$, 
& [$02/01$, $02/10$, $01/02$, $01/20$, $00/12$, $00/21$]; \\
 (iii) & $\cS_{p,q}=\alpha\sI+\beta\sJ+\gamma\sJ^2$ ($\{\alpha,\beta,\gamma\}=\{\frac17,\frac27,\frac47\}$),
& [$00/01$, $01/10$, $00/02$, $02/20$, $01/12$, $02/21$].
\end{tabular}
$$
The first two cases are uninteresting: (i) occurs when $(p,q)=(0,0),(1,\ast),(\ast,1)$ 
(where $\ast\ne 1$) so that $\cS_{p,q}$ simply represents a permutation
of vertex labels (cf.\,\cite[(3.2)]{NO2}); (ii) occurs when
$\cS_{p,q}(\Delta)$ represents the midpoint triangle, while the sides of 
$\cM^{\Mlbl}_{p,q}(\Delta)$ consists of the half sides of $\Delta$
when
$(p,q)=(0,\frac12),(\frac12,0)$. 
(Note: $\cS_{p,q}=\frac12(\sI+\sJ)$ never occurs).
However, (iii) yields geometrically nontrivial cases 
(as in Figure \ref{Fig4})
when 
$(p,q)=(\frac45,\frac23), (\frac15,\frac13), (\frac23,\frac13), (\frac13,\frac23),
(\frac13,\frac15), (\frac23,\frac45)$.
These are operations for Routh's triangles
discussed in \cite[Example 5.3]{NO2}.

\begin{figure}[htbp]
\begin{center}
\begin{minipage}{0.32\hsize}
\begin{center}
\igl{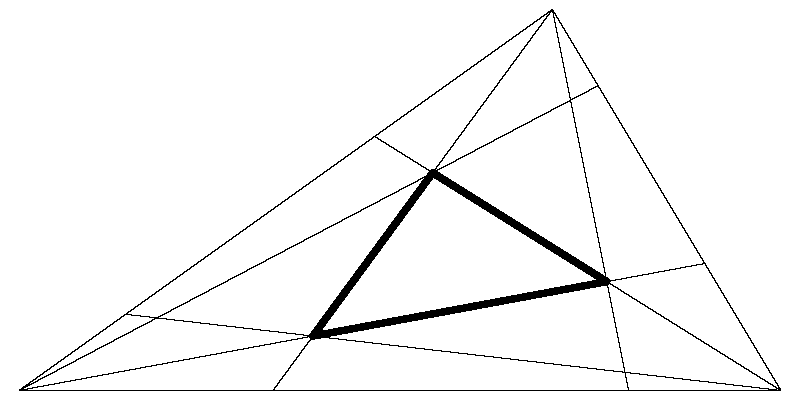} 
\caption*{$\cM^{00/01}_{\frac45,\frac23}(\Delta)=\cS_{\frac45,\frac23}(\Delta)$}
\end{center}
\end{minipage} 
\begin{minipage}[c]{0.32\hsize}
      \centering
$\quad$\igl{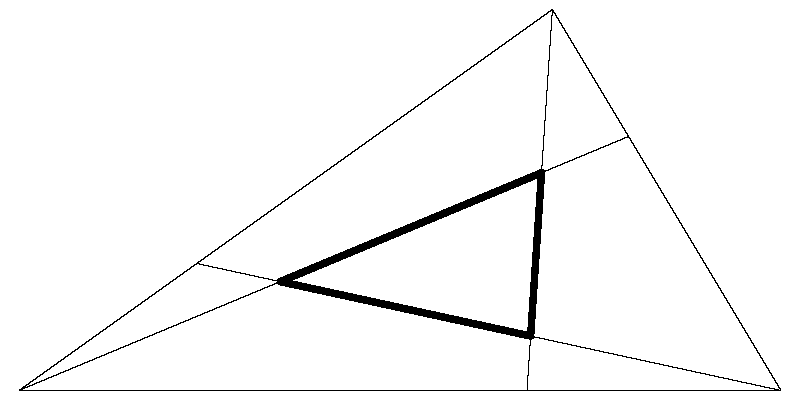} 
     \caption*{$\cM^{00/02}_{\frac23,\frac13}(\Delta)=\cS_{\frac23,\frac13}(\Delta)$}
      \label{tb:table2}
\end{minipage} 
\begin{minipage}[c]{0.32\hsize}
      \centering
$\quad$\igl{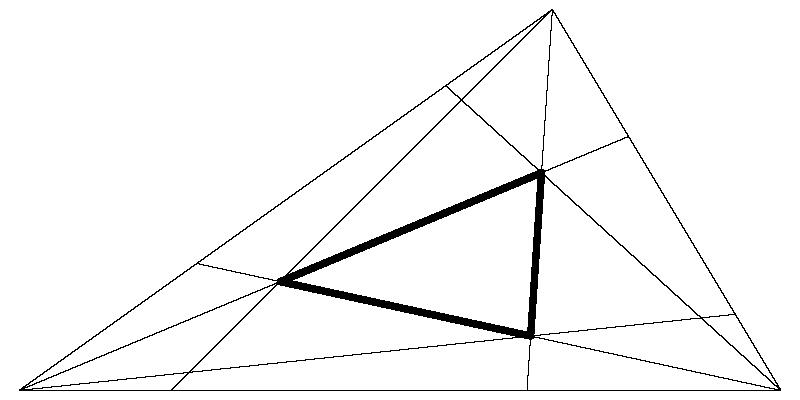}
    \caption*{$\cM^{01/10}_{\frac15,\frac13}(\Delta)=\cS_{\frac15,\frac13}(\Delta)$}
\end{minipage}  
\\
\medskip
\begin{minipage}[c]{0.32\hsize}
      \centering
\igl{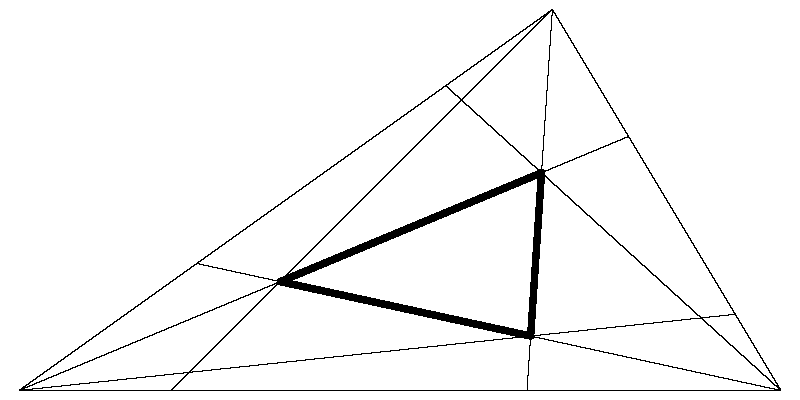}
     \caption*{$\cM^{02/21}_{\frac23,\frac45}(\Delta)=\cS_{\frac23,\frac45}(\Delta)$}
\end{minipage} 
\begin{minipage}[c]{0.32\hsize}
      \centering
$\quad$\igl{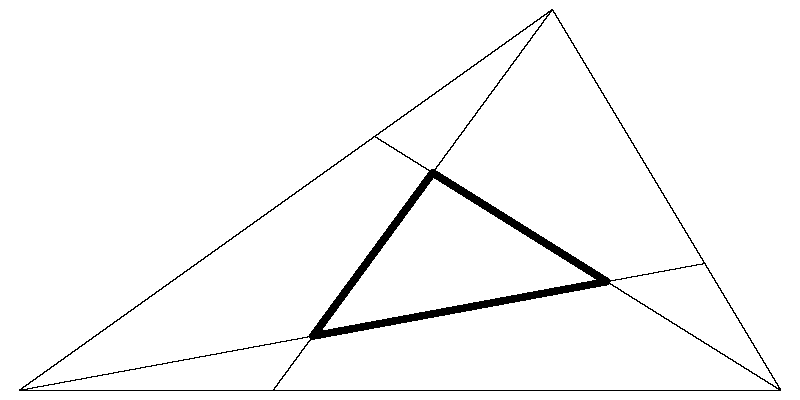} 
   \caption*{$\cM^{02/20}_{\frac13,\frac23}(\Delta)=\cS_{\frac13,\frac23}(\Delta)$}
      \label{tb:table2}
\end{minipage} 
\begin{minipage}[c]{0.32\hsize}
\centering
$\quad$\igl{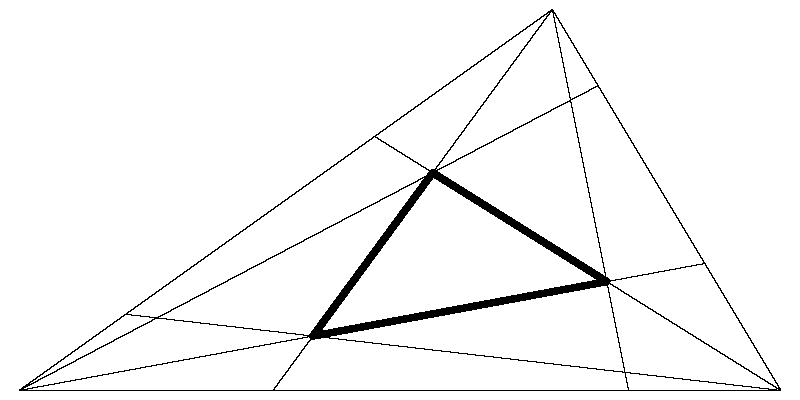} 
 \caption*{$\cM^{01/12}_{\frac13,\frac15}(\Delta)=\cS_{\frac13,\frac15}(\Delta)$}
\end{minipage}
\caption{
Type (iii) for $\cM^{\Mlbl}_{p,q}(\Delta)=\cS_{p,q}(\Delta)$
with $\Delta=\Delta(0)=(0,1,0.7+0.5i)$}
\label{Fig4}
\end{center}
\end{figure}

\end{Example}
%

\section{Shape space and B\'enyi-\'Curgus lifts}

One advantage of considering the (finite)
Fourier transforms of triangle triples (\ref{fourier})
is to enable us to catch directly
the {\it shape function} 
\begin{equation}
\label{shape-fn}
\psi(\Delta):=\frac{\psi_2(\Delta)}{\psi_1(\Delta)} \in \C\cup \{\infty\}=\mathbf{P}^1_\psi(\C)
\end{equation}
for triangle triples $\Delta$ without triple collision,
i.e., $\Delta\not\in \{(a,a,a)\mid a\in\C\}$.
The shape function was explicitly introduced by works of Hajja (e.g. \cite{H09})
and has been investigated by many authors including 
Nicollier \cite{Nic13}, B\'enyi-\'Curgus \cite{BC12}.
The idea of applying finite Fourier transformation to study polygon geometry
can be traced back to I.J.Schoenberg \cite{Sch50}.
The value of shape function $\psi(\Delta)$ represents the modulus
of shape 
(similarity class) of $\Delta$ as a triangle with vertices
labelled as vector components,
whereas the triple power $\psi^3(\Delta)$ represents the shape
of $\Delta$ as an oriented triangle with unlabelled vertices,
which is also useful in some geometrical problem (cf. e.g.,\,\cite{NO03}).
For our later discussions in \S 6, it is useful to set up the moduli 
space of triangles (with no triple collision) and their value spaces 
for $\psi,\psi^3$:

\begin{Construction}
\label{shape-sphere}
Write $(\C^3)^\flat:=\C^3-\{(a,a,a)\mid a\in\C\}$ for 
the collection of triangle triples with no triple collisions,
and 
consider the Fourier transform
$\Psi=(\psi_0,\psi_1,\psi_2)$ of (\ref{fourier}) 
as a vector valued function $(\C^3)^\flat\to\C^3$.
Noticing that $\psi_0(\Delta)$ concerns positioning
of (the centroid of) $\Delta$, we may regard 
the projection to the last two components 
as the classifying map to the
{\it space of translation classes of triangles} 
(written $C^\flat(\psi)$) 
in the form
$$
\mathrm{pr}:(\C^3)^\flat\twoheadrightarrow
C^\flat(\psi):=\left\{\bvect{\psi_2}{\psi_1}\in\C^2
\left\vert (\psi_1,\psi_2)\ne(0,0)  \right.\right\}.
$$ 
Here we mean by `triangle' an ordered triple
$(a,b,c)\in\C^3$ of three vertices 
admitting double collisions but no triple collisions. 
The shape function $\psi(\Delta)=\frac{\psi_2(\Delta)}{\psi_1(\Delta)}$
for $\Delta\in(\C^3)^\flat$ factors through the space $C^\flat(\psi)$
and terminates in 
$\mathbf{P}^1_\psi(\C)=\C\cup\{\infty\}$
which is naturally regarded as the {\it moduli space of 
similarity classes of triangles} (with labelled vertices).
We call  $\mathbf{P}^1_\psi(\C)$
the {\it shape sphere}.
Denote by $\mathbf{P}^1_{33}(\C)$ 
the quotient orbifold of 
$\mathbf{P}^1_\psi(\C)$
by the multiplication action of the cyclic group
$\{1,\omega,\omega^2\}$, 
where the subscript `$33$' indicates 
the two elliptic points of order 3 at $0,\infty$.
Note that the set of complex points 
$\mathbf{P}^1_{33}(\C)$ corresponds
bijectively to $\{\psi(\Delta)^3\mid \Delta\in(\C^3)^\flat\}$.
\end{Construction}

The relation of the above four spaces are summarized 
as in the commutative diagram
\begin{equation}
\label{shape-orbifold}
\begin{tikzcd}
(\C^3)^\flat \arrow[r, "\mathrm{pr}"] & C^\flat(\psi)
\arrow[rr, "\psi" ] \arrow[rrd, "\psi^3"' ] 
&  & {\mathbf{P}^1_\psi(\C)} \arrow[d] 
                                          \\
& &  & \mathbf{P}^1_{33}(\C).                           
\end{tikzcd}
\end{equation}

\begin{Note}
The space $\mathbf{P}^1_\psi(\C)$ is 
a complex analytic model of the {\it shape sphere}
appearing in the study of 3-body problem 
in celestial dynamics (see, e.g., \cite{MM15}).
\end{Note}

Since the operator $\cS[\eta,\eta']$ acts on $\Delta$ with 
componentwise multiplication by $(1,\eta',\eta)$ to $\Psi(\Delta)$
 (cf. (\ref{EtaTimesPsi})),
it acts on the values of shape function
$\psi(\Delta)=\psi_2(\Delta)/\psi_1(\Delta)$ as 
\begin{equation}
\label{eq5.1}
\psi(\cS[\eta,\eta'](\Delta))=\left(\frac{\eta}{\eta'}\right)\psi(\Delta).
\end{equation}

In Example \ref{Hajja}, we discussed Hajja's operator $\cH_s$ 
as a special case of generalized median operator 
$\cM^{00/12}_{0,1-s}$.
As shown in Example \ref{Hajja2}, it is written in 
Fourier parameters $(\eta,\eta')$ as 
\begin{equation} \label{matrixHajja}
\cH_s=\cM^{00/12}_{0,1-s}=\cS[s+\omega,s+\omega^2]
=W
\begin{pmatrix} 
1 & 0 & 0 \\ 0 & s+\omega^2 & 0 \\ 0 & 0 & s+\omega
\end{pmatrix}
W^{-1}.
\end{equation}
On the other hand, B\'enyi-\'Curgus \cite{BC12} introduces
an operator `$\cC_s$' (called the binary Ceva operator after
a seminal article \cite{G05}) for a real parameter $s$
closely related to $\cH_s$.
Although their formulation is given in slightly different language
of side length triples, we may extend the equivalent notion for
complex parameter $s\in\C$ as follows.

\begin{Definition}[binary Ceva operator]
For complex $s\in\C$, define the operator 
$\cC_s$ on triangle triples
by the following matrix expression:
\begin{equation}
\label{matrixBC}
\cC_s=W
\begin{pmatrix} 
1 & 0 & 0 \\ 0 & 0 & s+\omega \\ 0 & s+\omega^2 & 0 
\end{pmatrix}
W^{-1}.
\end{equation}
In other words, $\cC_s$ is defined so as to operate 
on the Fourier parameters of triangle triples by:
$$
\psi_0(\cC_s(\Delta))=\psi_0(\Delta),\quad
\psi_1(\cC_s(\Delta))=(s+\omega)\psi_2(\Delta),\quad
\psi_2(\cC_s(\Delta))=(s+\omega^2)\psi_1(\Delta).
$$
\end{Definition}
%

It is not difficult to see that either of 
$\cH_s$, $\cC_s$ preserve $(\C^3)^\flat\subset\C^3$
if and only if $s\ne \rho,\rho^{-1}$.
Assuming this condition, let us write $\bar\cH_s$, $\bar\cC_s$ for the operators
on $C^\flat(\psi)$ induced respectively from $\cH_s,\cC_s$.
These are also expressed as matrices
acting on $\bvect{\psi_2}{\psi_1}\in C^\flat(\psi)$ 
in the form
\begin{equation}
\label{eq5.9}
\bar\cH_s=\bmatx{s+\omega}{0}{0}{s+\omega^2},\quad
\bar\cC_s=\bmatx{0}{s+\omega^2}{s+\omega}{0}.
\end{equation}
Below we use the quantity $\xi_s:=\frac{s+\omega}{s+\omega^2}$
to denote the multiplier factor for $\cH_s$ on the
shape function. Note that the condition $s\ne \rho,\rho^{-1}$ is equivalent
to $\xi_s\in\C^\times$.
The following numerical identities easily derived from definitions
relate the shapes of
$\Delta$, $\cH_s(\Delta)$, $\cC_s(\Delta)$, $\cH_{1-s}(\Delta)$ 
and of $\cC_{1-s}(\Delta)$:
\begin{equation}
\label{eq5.10}
\psi(\cH_s(\Delta))=\xi_s\cdot \psi(\Delta), \quad
\psi(\cC_s(\Delta))=\xi_s^{-1}\psi(\Delta)^{-1}, \quad
\xi_s^{-1}=\xi_{1-s} 
\qquad (s\ne \rho,\rho^{-1}).
\end{equation}
We are then led to Proposition \ref{BClift} below which
translates selected geometrical relations found 
by B\'enyi-Curgus \cite{BC12} into our language
of operations $\cH_s$, $\cC_s$.
We recall 
that two triangles $\Delta=(a,b,c)$ and 
$\Delta'=(a',b',c')$ are called 
{\it directly similar} (resp. {\it reversely similar}) 
in \cite[p.378]{BC12} if $(a',b',c')$ is similar to
$(a,b,c)$ (resp. $(a,c,b)$) as oriented triangles 
\underline{without} labeling of vertices. 
We will write $\Delta\stackrel{\mathtt{dr}}{\sim}\Delta'$
(resp. $\Delta \stackrel{\mathtt{rv}}{\sim} \Delta'$)
for the direct similarity (resp. reverse similarity) which is 
equivalent to $\psi(\Delta)^3=\psi(\Delta')^{3}$ 
(resp. $\psi(\Delta)^3=\psi(\Delta')^{-3}$).

\begin{Proposition} 
\label{BClift}
Notations being as above, the following formulas and statements hold
for $s,r,u\in\C\setminus\{\rho,\rho^{-1}\}$.
\begin{enumerate}[label=(\roman*),font=\upshape]
\item[(i)] $\bar\cH_s\circ\bar\cH_s=-\bar\cC_{1-s}\circ \bar\cC_s$.
\item[(ii)] $\bar\cC_s\circ\bar\cC_s=(s^2-s+1) \,\mathrm{id}$.
\item[(iii)] For each $\Delta\in(\C^3)^\flat$, we have
$\cC_s(\Delta)\stackrel{\mathtt{rv}}{\sim} \cH_s(\Delta)$. 
\item[(iv)] For each $\Delta\in(\C^3)^\flat$, 
$\cC_s\circ\cC_r(\Delta)\stackrel{\mathtt{dr}}{\sim}\cC_u(\Delta)$
if and only if $(\xi_r\xi_u)^3\cdot\psi(\Delta)^6=\xi_s^3$.
\item[(v)] $\cC_s\circ\cC_r(\Delta)
\stackrel{\mathtt{rv}}{\sim}\cC_u(\Delta)$ 
for all $\Delta\in (\C^3)^\flat$ if and only if
$(\xi_s\xi_u)^3=\xi_r^3$.
\item[(vi)]
If $\Delta\stackrel{\mathtt{rv}}{\sim}\Delta'$, then
$\cC_s(\Delta)\stackrel{\mathtt{rv}}{\sim} \cC_{1-s}(\Delta')$.
\end{enumerate}
\end{Proposition}

\begin{proof}
(i), (ii) follow from matrix computations by (\ref{eq5.9}),
and (iii) follows from (\ref{eq5.10}) at once.
The values of shape function at $\cC_s\circ\cC_r(\Delta)$ and at $\cC_u(\Delta)$ 
are respectively $\xi_s^{-1}\xi_r\cdot \psi(\Delta)$ and $\xi_u^{-1}\psi(\Delta)^{-1}$.
This proves (iv). Next, by a remark preceding the proposition, 
$\cC_s\circ\cC_r(\Delta)$ is reversely similar 
to $\cC_u(\Delta)$ iff $(\xi_s^{-1}\xi_r\psi(\Delta))^3=(\xi_u^{-1}\psi(\Delta)^{-1})^{-3}$,
from which the assertion (v) follows 
by cancelling out the common factor $\psi(\Delta)^3$.
Finally, the assumption of (vi) is equivalent to $\psi(\Delta)^3\psi(\Delta')^3=1$.
On the other hand, the entry quotients of the identities 
$\bar\cC_s(\Delta)=\bvect{(s+\omega^2)\psi_1(\Delta)}{(s+\omega)\psi_2(\Delta)}$ and
$\bar\cC_{1-s}(\Delta')=\bvect{(-s-\omega)\psi_1(\Delta')}{(-s-\omega^2)\psi_2(\Delta')}$
yield $\psi(\cC_s(\Delta))\cdot \psi(\cC_{1-s}(\Delta'))=\psi(\Delta)^{-1} \psi(\Delta')^{-1}$.
Thus, we conclude $\psi(\cC_s(\Delta))^3 \psi(\cC_{1-s}(\Delta'))^3=1$,
that is, $\cC_s(\Delta)\stackrel{\mathtt{rv}}{\sim} \cC_{1-s}(\Delta')$ as desired.
\end{proof}

\begin{Note}
The assertions (i), (ii), (iv), (v)  and (vi) of Proposition \ref{BClift} correspond 
respectively to a property at line $-5$ in p.379, 
Proposition 9.2, Corollary 9.8, Theorem 9.9 and Proposition 9.1
of \cite{BC12}.
The assertion (iii) is originally a source defining property for 
the binary Ceva operator \cite[p.379]{BC12}.
A binary operation $\ast$ defined by the identity $\xi_s\xi_{s'}=\xi_{s\ast s'}$ 
gives a commutative group structure on $\{s\in\mathbf{P}^1(\C)\mid s\ne \rho,\rho^{-1}\}$
which is equivalent to
an operation $\square$ on $\R\cup \{\infty\}$ introduced in \cite{BC12} 
and to the additive operation $[+]$ on $\mathbf{P}^1(\C)-\{\rho,\rho^{-1}\}$
studied in \cite{NO2} after suitable variable changes.
\end{Note}

\begin{Note}
Proposition \ref{BClift} (i), (ii) can be rephrased at the level of 
operators on $(\C^3)^\flat$ respectively as:

(i) $\cH_s\circ\cH_s=\cS[-1,-1]\circ\cC_{1-s}\circ \cC_s$;

(ii) $\cC_s\circ\cC_s=\cS[s^2-s+1,s^2-s+1]$.
\end{Note}

\begin{Note}
The above usage of `reverse similarity' differs from
the one employed in \cite{NO03} where it  
meant $\Delta\stackrel{\mathtt{dr}}{\sim}
\overline{\Delta'}$ (mirror image of $\Delta'$)
that is sometimes called anti-similarity.
\end{Note}


\section{Tracing orbits of triangles}

Since our operators $\cS[\eta,\eta']$ are realized as linear actions
on $\C^3$ that fix $(1,1,1)$ by Proposition \ref{Spq:eqn} (i), 
they commute with every complex 
affine transformation of triangles, in other words, 
$\cS[\eta,\eta']$ commutes with any mapping of the form
$(a,b,c)\mapsto (f(a),f(b),f(c))$ where
$f:z\mapsto \lambda z+\nu$ ($\lambda,\nu\in\C$).
This is not always the case for {\it real} affine transformations. 
We first begin with the following simple lemma.

\begin{Lemma} \label{RealAffine}
Let $(\eta,\eta')\in\C^2$.
The operation $\cS[\eta,\eta']$ commutes with the real affine transformations
of triangles if and only if 
$\bar\eta =\eta'$, i.e., $\eta$ and $\eta'$ are complex conjugate to each other.
\end{Lemma}

\begin{proof}
Recall that any real affine transformation of the complex plane $\C$ can be
written as $f_{\lambda,\mu,\nu}(z)=\lambda z+\mu\bar z+\nu$ with
$\lambda,\mu,\nu\in\C$.
Given a triangle triple $\Delta=(a,b,c)$ and $f=f_{\lambda,\mu,\nu}$, write
$f(\Delta):=(f(a),f(b),f(c))$ for the image of $\Delta$ by $f$.
Then, one computes 
$$
\cS[\eta,\eta']\left( f_{\lambda,\mu,\nu}\tvect{a}{b}{c} \right)=
f_{\lambda,\mu,\nu}\left(\cS[\eta,\eta']\tvect{a}{b}{c}\right)
+\mu\cdot \left(\cS[\eta,\eta']-\overline{\cS[\eta,\eta']}\right)
\tvect{a}{b}{c}.
$$
The commutativity of $\cS[\eta,\eta']$ and $f_{\lambda,\mu,\nu}$ holds
if and only if  $\mu=0$ (i.e., $f_{\lambda,\mu,\nu}$ is complex affine)
or 
$\cS[\eta,\eta']=\frac{1}{3}(1+\eta+\eta')\sI
+\frac{1}{3}(1+\eta\omega+\eta'\omega^2)\sJ
+\frac{1}{3}(1+\eta\omega^2+\eta'\omega)\sJ^2
$ is in $M_3(\R)$.
The latter condition is easily seen to be equivalent to 
$\bar\eta=\eta'$.
\end{proof}

In \cite{NO2}, we called $\cS[\eta,\eta']\in\bbS$ 
an area-preserving operator if the associated parameters $\eta,\eta'\in\C$ satisfy 
$|\eta|=|\eta'|=1$. 
The set of area-preserving operators forms a compact multiplicative
torus in $\GL_3(\C)$. 
Since $\cS[\eta,\eta']^k=\cS[\eta^k,{\eta'}^k]$ ($k\in \Z$),
iteration of area-preserving operators 
can be interpolated by one-parameter family of the form
$$
\{\cS[e^{2\pi i mt},e^{2\pi i nt}]\}_{t\in \R}.
$$
We are particularly interested in the case where three
vertices move along a single closed orbit cyclically replacing 
positions of each other after $t\mapsto t+\varpi/3$ 
so that the total motion is left invariant after $t\mapsto t+\varpi$.
Note that, in this situation, 
we may assume $\varpi=1$ and $m,n$ are coprime integers
without loss of generality.
Taking this into accounts, we are led to start with a more general setup:
Recall from \ref{shape-sphere} that $(\C^3)^\flat$ denotes 
the collection of triangle triples with no triple collisions.
Suppose we are given $\Delta\in(\C^3)^\flat$ and 
two continuous functions $\eta,\eta':\R\to(\R/\Z\to) \C$ (with period 1).
We shall consider the periodic maps $\R\to\R/\Z\to (\C^3)^\flat$ in the form
$$
\Delta(t)=\cS[\eta(t),\eta'(t)](\Delta) \text{ or } \cM^\Mlbl[\eta(t),\eta'(t)](\Delta).
$$
Note that generally $\Delta(0)$ may not be the same as the initial $\Delta$ 
and that $\Delta(t)$ may degenerate at some $t$ even 
if $\Delta$ is given as a non-degenerate triangle. 
The family $\{\Delta(t)\}_t$ will be called {\it collision-free} 
if, for every $t\in\R$, $\Delta(t)$ is a (degenerate or non-degenerate)
triangle with three distinct vertices.
We sometimes regard the time parameter $t\in\R$ also as $t\in\R/\Z$
when no confusion could occur.

\begin{Definition}
\label{DefSingleOrbit}
Notations being as above, we say the family $\{\Delta(t)\}_{t\in\R/\Z}$ 
to have a single tracing orbit in ascending (resp. descending) order,
if 
$ \sJ\, \Delta(t) =\Delta(t+\frac{1}{3}) $ 
(resp. $ = \Delta(t-\frac{1}{3}) $). 
\end{Definition}

If $\{\Delta(t)\}_t$ has a single tracing orbit in ascending order, and 
$\Delta(t)$ is written as $(a_0(t),a_1(t),a_2(t))$,  then,
$a_0(t)=a_2(t+\frac13)=a_1(t+\frac23)=a_0(t+1)$ for all $t\in\R$.
We may interpret a collision-free family with this property as
a motion of three particles $a_0,a_1,a_2$ 
moving along a single 
closed orbit so that they trace each other chronologically
with $a_0\to a_1\to a_2\to a_0$.

\begin{Proposition} \label{singleorbit}
Let $\Delta\in(\C^3)^\flat$ and 
$\eta,\eta':\R/\Z \to \C$ be continuous functions with period 1.
\begin{enumerate}[label=(\roman*),font=\upshape]
\item
$\{\cS[\eta(t),\eta'(t)](\Delta)\}_t $ has a single tracing orbit in ascending 
(resp. descending) order
if and only if
\begin{align*} 
&\qquad\eta(t+\frac13)=\eta(t)\,\omega^{-1}, \
\eta'(t+\frac13)=\eta'(t)\,\omega \quad(t\in\R)
\\
&\left( \text{ resp. } 
\eta(t-\frac13)=\eta(t)\, \omega^{-1},  \ 
\eta'(t-\frac13)=\eta'(t)\,\omega  \quad(t\in\R)
\right)
\end{align*}
holds.
\item Let $\Mlbl$ be a label for generalized median operators. Then,
$\{\cM^\Mlbl[\eta(t),\eta'(t)](\Delta)\}_t$
has a single tracing orbit in ascending 
(resp. descending) order
if and only if 
$\tilde\eta(t):=\eta(t)-\omega^{\ttx-\ttw}$,
$\tilde\eta'(t):=\eta'(t)-\omega^{\ttw-\ttx}$
satisfy
\begin{align*} 
&\qquad
\tilde\eta(t+\frac13)=\tilde\eta(t)\,\omega^{-1}, \
\tilde\eta'(t+\frac13)=\tilde\eta'(t)\,\omega  \quad(t\in\R).
\\
&\left( \text{ resp. } 
\tilde\eta(t-\frac13)=\tilde\eta(t)\, \omega^{-1},  \ 
\tilde\eta'(t-\frac13)=\tilde\eta'(t)\,\omega  \quad(t\in\R).
\right)
\end{align*}
\end{enumerate}
\end{Proposition}

\begin{proof}
(i) follows immediately from (\ref{eq2.1}).
To prove (ii), we make use of Corollary \ref{MStran} to express
$\cM^\Mlbl[\eta(t),\eta'(t)](\Delta)$ as 
$\cS[\eta_0(t),\eta_1(t)]$.
Then, apply (i) for the latter form.
\end{proof}

Let us look more closely at the tracing orbit in relation with 
the shape sphere $\mathbf{P}^1_\psi(\C)$ 
introduced in Definition \ref{shape-sphere}.
Write $\Conf^3(\C)$ (called the {\it configuration space}) 
for the collection of collision-free triples in $(\C^3)^\flat$, i.e., 
$$
\Conf^3(\C):=\{(a,b,c)\in\C^3\mid a\ne b\ne c\ne a\}
\subset (\C^3)^\flat.
$$
Then, it is easy to see that 
the shape function $\psi=\frac{\psi_2}{\psi_1}$ maps
$\Conf^3(\C)$ onto the open locus 
$\mathbf{P}^1_\psi(\C)-\{1,\omega,\omega^2\}$ of the shape sphere
$\mathbf{P}^1_\psi(\C)$.
Given a collision-free single tracing orbit 
$\{\Delta(t)\}_t$ $\subset \Conf^3(\C)$ either in
ascending or descending order, the image $\{\psi(\Delta(t))\}_{0\le t\le 1}$ 
move on a closed curve on $\mathbf{P}^1_\psi(\C)-\{1,\omega,\omega^2\}$.
More precisely, during the process starting/ending 
at the point $\psi_0:=\psi(\Delta(0))=\psi(\Delta(1))$, it passes
two distinguished points $\psi(\Delta(\frac13))=\omega \psi_0$,
$\psi(\Delta(\frac23))=\omega^2 \psi_0$ 
(resp. $\psi(\Delta(\frac13))=\omega^2 \psi_0$,
$\psi(\Delta(\frac23))=\omega \psi_0$)
if $\{\Delta(t)\}_t$ moves in ascending order (resp. in descending order).
Thus, in view of the diagram (\ref{shape-orbifold}), the image of 
$\{\psi(\Delta(t))^3\}_{0\le t\le \frac13}$ forms a closed curve
on $\mathbf{P}^1_{33}(\C)-\{1\}$. 
It is worth noting that the (existence and) 
classification of the tracing orbits $\{\Delta(t)\}_t$ 
sharing a same closed curve on $\mathbf{P}^1_{33}(\C)-\{1\}$
is available as below, where we employ the convention $1/0=\infty$.

\begin{Proposition}
Let $\eps\in\{\pm 1\}$ and let
$\gamma: \R\to \mathbf{P}^1_\psi(\C)-\{1,\omega,\omega^2\}$
be a non-constant continuous map 
such that 
$\gamma(0)\not\in \{0,\infty\}$ and 
$\gamma(t+\frac13)=\omega^\eps
\gamma(t)$ for all $t\in\R$.
Set $\xi(t):=\gamma(t)/\gamma(0)$.
\begin{enumerate}[label=(\roman*),font=\upshape]
\item[(i)] There exists a collision-free single tracing orbit $\{\Delta(t)\}_t$
with $\psi(\Delta(t))=\gamma(t)$
in ascending $($resp. descending$)$ order when $\eps=1$ 
$($resp. $\eps=-1)$, if and only if 
$\xi(t)$ 
can be written in the form 
$\displaystyle
\xi(t)=\frac{\eta_2(t)}{\eta_1(t)}
$
with
$\eta_r:\R\to \C$ satisfying
$\eta_r(t+\frac13)=\omega^{\eps r} \eta_r(t)$ $(r=1,2)$.
\item[(ii)] In particular, 
if $\gamma(t)\ne \infty$ 
$(\forall \, t)$
then
(i) is the case with $\eta_2(t)=e^{2\pi i \eps t} \xi(t)$,
$\eta_1(t)=e^{2\pi i \eps t}$.
\item[(iii)]
Suppose (i) is the case. Then, every possible tracing orbit 
$\{\Delta(t)\}_t$ sharing the same closed curve
$\gamma(t)$ as $\psi(\Delta(t))$
is obtained as
$$
\Delta(t)=W 
\begin{pmatrix} 
1 & 0 & 0 \\ 0 & \eta_1(t)\mu(t) & 0 \\ 0 & 0 & \eta_2(t)\mu(t)
\end{pmatrix}
W^{-1}
\begin{pmatrix} 
a_0 \\ b_0 \\ c_0
\end{pmatrix}
$$
where $\mu:\R\to\C^\times$ is a continuous function with 
period $\frac13$, and $\Delta_0=(a_0,b_0,c_0)$ is a
triangle triple with $\psi(\Delta_0)=\gamma(0)$.
\end{enumerate}
\end{Proposition}

\begin{proof}
The assumption on $\gamma(0)$ tells that the value
$\lambda_0:=-\frac{\gamma(0)\omega-1}{\gamma(0)-\omega}$ is contained
in $\C-\{0,1,\rho^{\pm 1}\}$ so that $\Delta_0:=(0,1,\lambda_0)\in\Conf^3(\C)$ 
forms a (non-equilateral) triangle triple with $\psi(\Delta_0)=\gamma(0)$.

(i) The `if'-part is already shown in Proposition \ref{singleorbit}, so we here
show the `only if' part. 
Suppose a collision-free single tracing orbit
$\{\Delta(t)\}_t$ exists and
consider the behavior of $\psi_r(\Delta(t))$ $(r=1,2)$. 
Since 
$\gamma(0)=\psi(\Delta(0))=\psi_2(\Delta(0))/\psi_1(\Delta(0))\not\in\{0,\infty\}
$ by assumption, 
we can define for each $r\in\{1,2\}$ a function
$\eta_r:\R\to\C$ by
$\eta_r(t):=\psi_r(\Delta(t))/\psi_r(\Delta(0))$. 
This shows that $\Psi(\Delta(t))=W^{-1}\Delta(t)$
is of the form $\mathrm{diag}(1,\eta_1(t),\eta_2(t)) \Psi(\Delta(0))$.
Then the prescribed property 
$\Delta(t+\frac{\eps}{3})= \sJ\, \Delta(t)$
(Definition \ref{DefSingleOrbit}) implies the required 
properties for $\eta_r(t+\frac13)$ $(r=1,2)$.

(ii) This is immediate after (i).

(iii) Fix a collision-free tracing orbit $\{\Delta(t)\}_t$ with parameter functions
$\eta_1(t),\eta_2(t)$ as in
(i), and pick any other such a family $\{\Delta'(t)\}_t$ with another set of parameter
functions $\eta_1'(t), \eta_2'(t)$. Then, 
$\frac{\eta_2(t)}{\eta_1(t)}=\xi(t)=\frac{\eta_2'(t)}{\eta_1'(t)}$
for all $t\in\R$. Note here that $\xi(t)=0$ (resp. $=\infty$) if and only if
$\eta_1\eta_1'\ne 0$, $\eta_2=\eta_2'=0$ 
(resp. $\eta_2\eta_2'\ne 0$, $\eta_1=\eta_1'=0$).
So the continuity of the map $\xi$ of $\R$ into the Riemann sphere 
$\C\cup\{\infty\}$ enables us to define a continuous function
$\mu:\R\to \C^\times$ by
$\mu(t):=\frac{\eta_1'}{\eta_1}(=\frac{\eta_2'}{\eta_2})$ when $\xi(t)\in\C^\times$, 
$\mu(t):=\frac{\eta_1'}{\eta_1}$ when $\xi(t)= 0$ and
$\mu(t):=\frac{\eta_2'}{\eta_2}$ when $\xi(t)=\infty$.
The periodic property $\mu(t+\frac13)=\mu(t)$ is a consequence of
$\eta_r(t+\frac13)=\omega^{\eps r} \eta_r(t)$
and $\eta_r'(t+\frac13)=\omega^{\eps r} \eta_r(t)'$ $(r=1,2)$.
The vector expression in (iii) follows from the procedure 
described in the above proof of (i) with $\Delta_0=(a_0,b_0,c_0)
=W \mathrm{diag}(1,\eta_1'(0),\eta_2'(0)) W^{-1} \Delta'(0)$,
that is, $\psi(\Delta_0)=\xi(0)\cdot\psi(\Delta'(0))=\xi(0)\gamma(0)=\gamma(0)$.
This completes the proof.
\end{proof}

Now, let us turn back to the area-preserving parameters 
$\eta(t)=e^{2\pi imt}$, $\eta'(t)=e^{2\pi i nt}$ with
coprime integers $m,n\in\Z$ and examine some 
typical cases.

\begin{Example} \label{ExSmn}
Let $\Delta$ be a triangle triple and $m,n$ coprime integers.
By Proposition \ref{singleorbit} (i),
the family 
$$
\{\Delta(t)=\cS[e^{2\pi imt}, e^{2 \pi in t}](\Delta)\}_{t\in\R}
$$
has a single tracing orbit if $m+n\equiv 0 \pmod 3$.
The vertices move in ascending (resp. descending) 
order if $m\equiv 2\pmod 3$ (resp. $m\equiv 1\pmod 3$).
\end{Example}

\begin{Example}[Steiner ellipse] \label{ExSteiner}
The special case $m=-1,n=1$ of Example \ref{ExSmn}
is  
$$
\Delta(t)=\cS[e^{-2\pi it}, e^{2\pi i t}](\Delta).
$$
In this case, starting from $\Delta(0)=\Delta$, the vertices 
of a triangle move on an ellipse with sides tangent 
to an interior ellipse (Figure \ref{Fig5}). 
For an easy proof for the case $\Delta$ is non-degenerate, 
one can apply Lemma \ref{RealAffine} 
to deform $\Delta$ to the equilateral triangle $(1,\omega,\omega^2)$
in real affine geometry.
If $\Delta=(0,1,u+v\sqrt{-1})$, then the circumscribed ellipse 
has the following equation in XY-coordinates of $\C$.
$$
v^2 \left(X-\frac{1+u}{3}\right)^2
+(v-2uv)\left(X-\frac{1+u}{3}\right)\left(Y-\frac{v}{3}\right)
+(1-u+u^2) \left(Y-\frac{v}{3}\right)^2=\frac{v^2}{3}.
$$
\begin{figure}[h]
\begin{center}
\begin{tabular}{c}
\qquad 
\ig{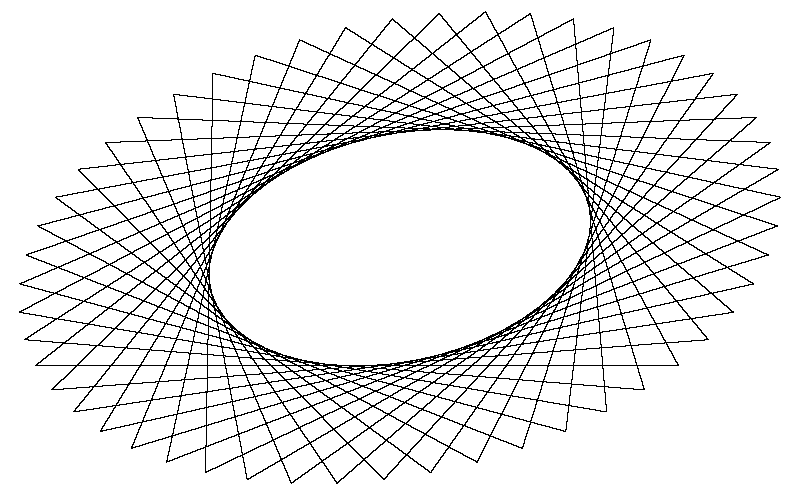}
\end{tabular}
\caption{
$\{\cS[e^{-2\pi i k/17}, e^{2\pi i k/17}](\Delta)\}_{k\ge \Z}$
for $\Delta=\Delta(0)=(0,1,0.7+0.5i)$
}
\label{Fig5}
\end{center}
\end{figure}
\end{Example}

\begin{Example} 
The following three collections of figures
(Figure \ref{Fig6}, \ref{Fig7}, \ref{Fig8}) illustrate
the family 
$\cS[e^{2\pi it}, e^{2 \pi i n t}](\Delta)$ for $n\equiv 2 \pmod 3$,
$\cS[e^{2\pi i m t}, e^{2\pi i t}](\Delta)$ for $m\equiv 2\pmod 3$
and some other types from Example \ref{ExSmn} respectively. 
We start from $\Delta=\Delta(0)=(0,1,0.7+0.5i)$.

\begin{figure}[h]
\begin{center}
\begin{tabular}{cc}
\quad
\igg{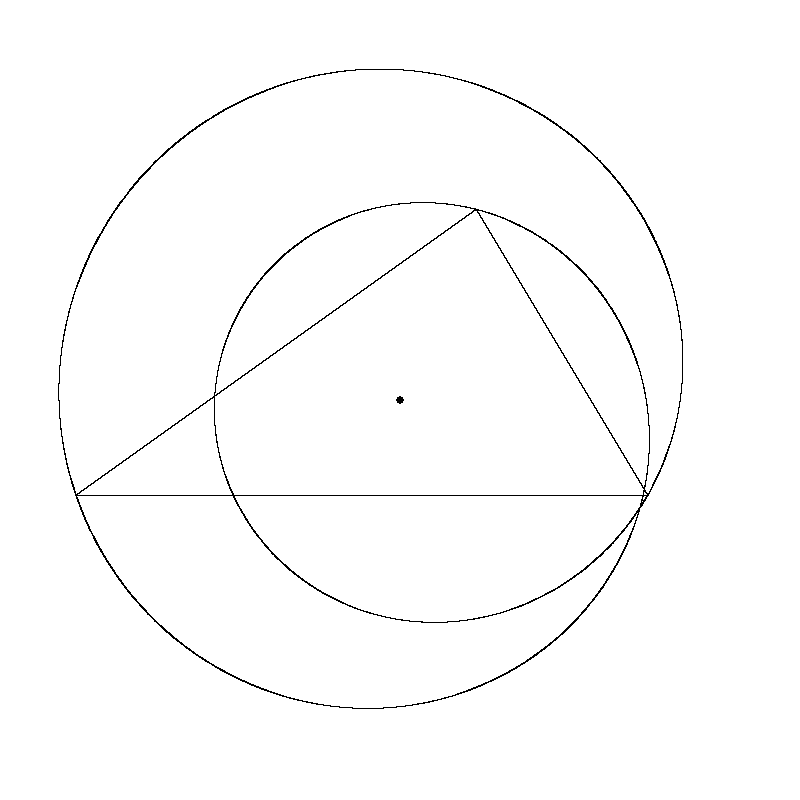} & 
\qquad 
\igg{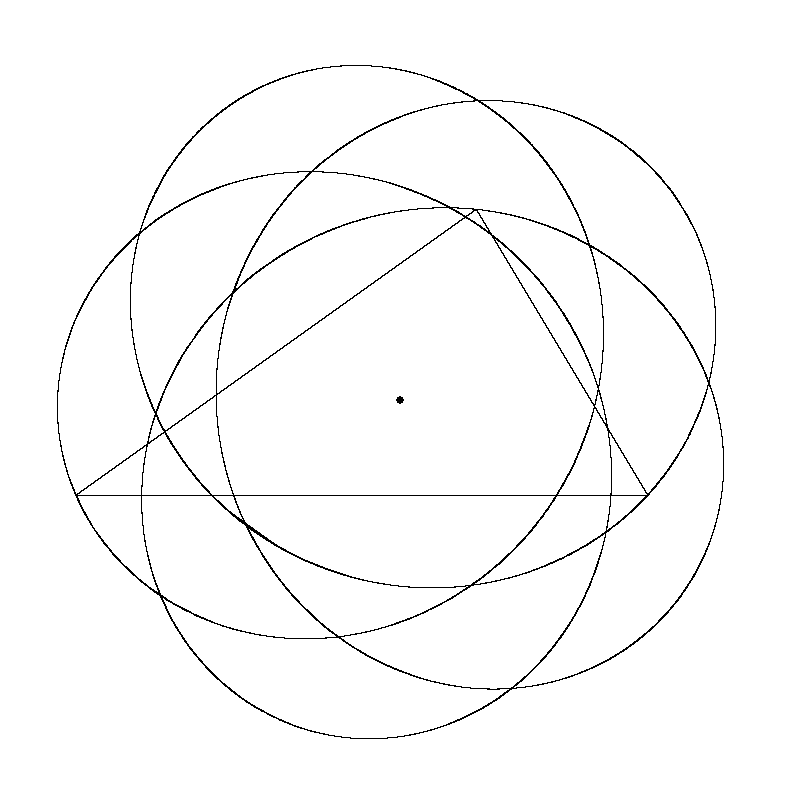}
\end{tabular}
\caption{
$\{\cS[e^{2\pi i t},e^{4\pi i t}, ](\Delta)\}_{t\in\R}$ and
$\{\cS[e^{2\pi i t},e^{-8\pi i t}](\Delta)\}_{t\in\R}$
}
\label{Fig6}
\end{center}
\end{figure}

\begin{figure}[h]
\begin{center}
\begin{tabular}{cc}
\quad
\igg{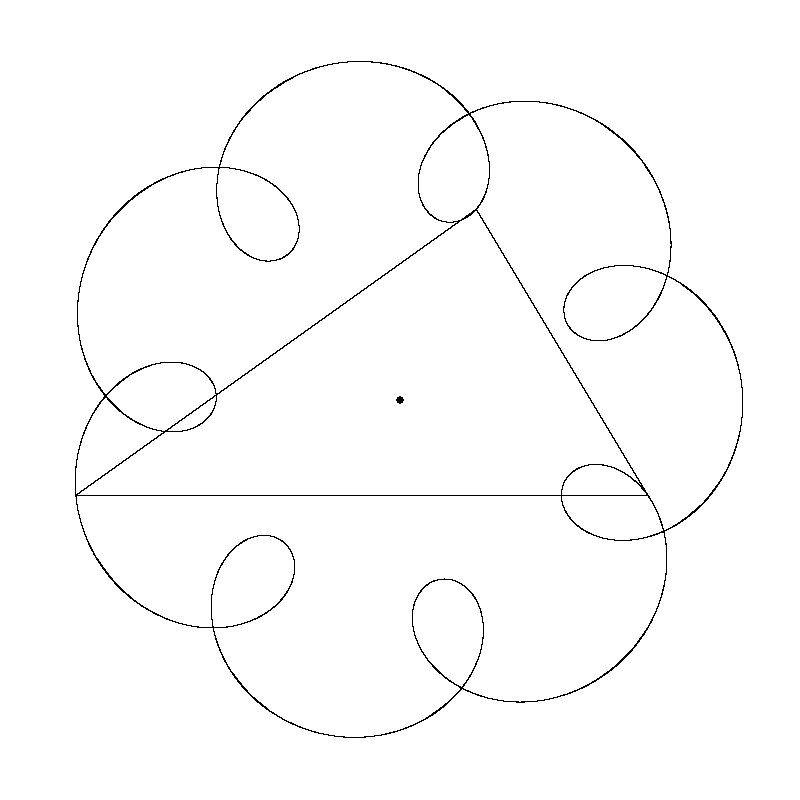} & 
\qquad 
\igg{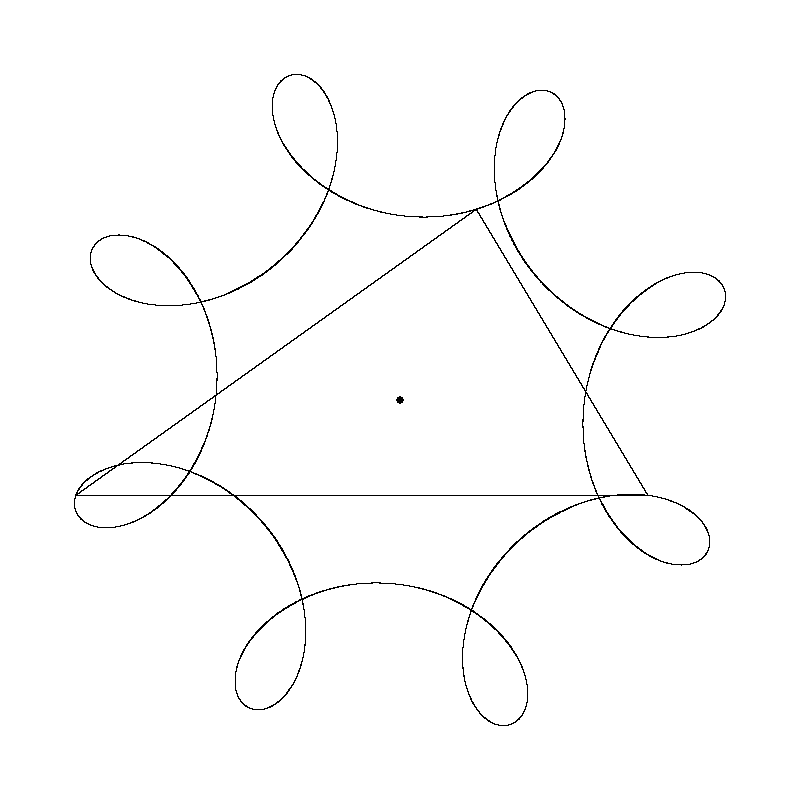}
\end{tabular}
\caption{
$\{\cS[e^{16\pi i t},e^{2\pi i t}](\Delta)\}_{t\in\R}$ and
$\{\cS[ e^{-14\pi i t},e^{2\pi i t}](\Delta)\}_{t\in\R}$
}
\label{Fig7}
\end{center}
\end{figure}

\begin{figure}[h]
\begin{center}
\begin{tabular}{cc}
\quad
\igg{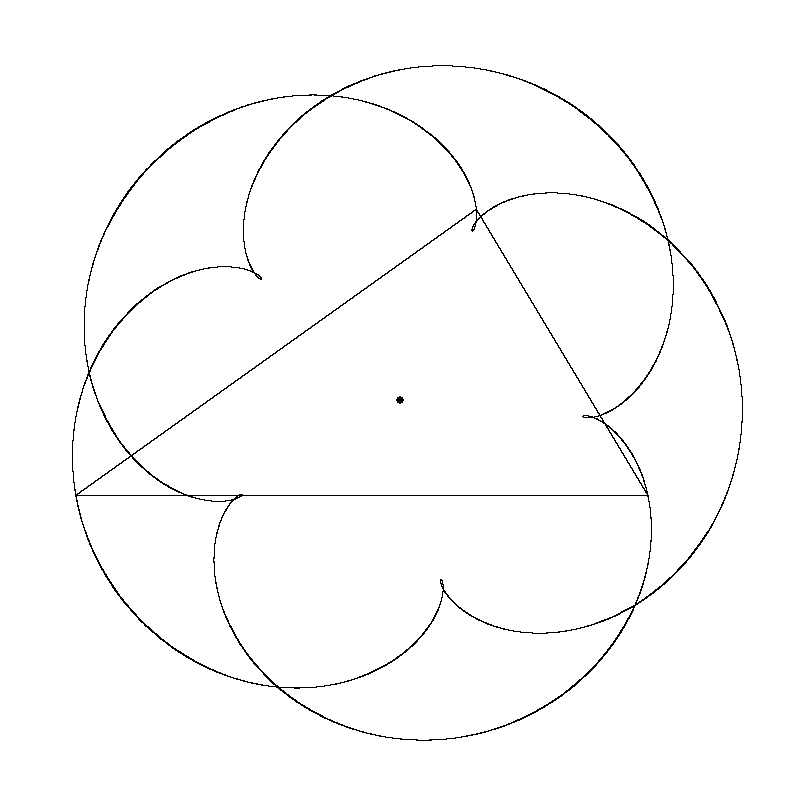} & 
\qquad 
\igg{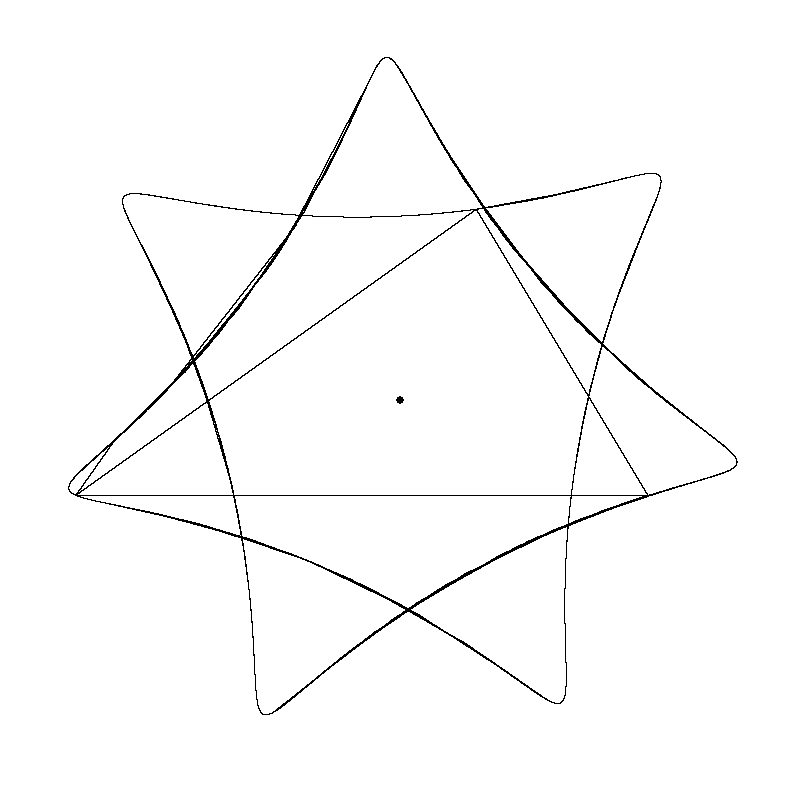}
\end{tabular}
\caption{
$\{\cS[ e^{14\pi i t},e^{4\pi i t}](\Delta)\}_{t\in\R}$ and
$\{\cS[e^{-10\pi i t},e^{4\pi i t} ](\Delta)\}_{t\in\R}$
}
\label{Fig8}
\end{center}
\end{figure}

\end{Example}

\begin{Example}[Median orbits]
By Proposition 
\ref{singleorbit} (ii), the median triangle family
$$
\Bigl\{\cM^{0\ttx/01}[\eta(t),\eta'(t)](\Delta)\Bigr\}_{t\in\R}
$$
along with 
$\eta(t)=e^{2\pi i m t}+\omega^\ttx$, 
$\eta'(t)=e^{2\pi i n t}+\omega^{-\ttx}$
($m+n\equiv 0\pmod 3$, $\ttx\in\Z/3\Z$)
has a single tracing orbit.
The following figure (Figure \ref{Fig9})
starts from $\Delta=(0,1, 0.7+0.5i)$,
$\Delta(0)=(\frac45+\frac13 i, \frac{7}{30}+\frac16 i, \frac23)$.
According to (\ref{eq2.8}), the orbit is independent of the choice of
$\ttx\in\Z/3\Z$.  
We also observe that it is similar to the orbit 
$\{\cS[e^{-10\pi i t},e^{4\pi i t} ](\Delta)\}_{n\in\R}$
illustrated in the previous example.
\begin{figure}[h]
\begin{center}
\begin{tabular}{c}
\quad
\igg{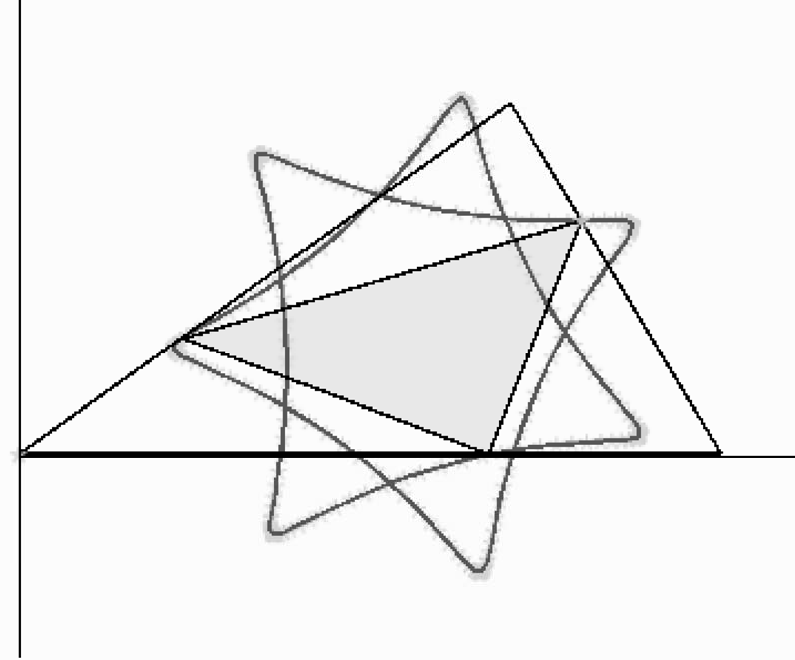}
\end{tabular}
\caption{$\Delta$ and
$\{\cM^{0\ttx/01}[e^{-10\pi i t}+\omega^\ttx,e^{4\pi i t}+\omega^{-\ttx}](\Delta)\}_{t\in\R}$
}
\label{Fig9}
\end{center}
\end{figure}
\end{Example}

It is not necessary for us to 
persist in area-preserving parameters 
in Proposition \ref{singleorbit}.

Simple linear sums of $e^{2\pi i m t}$ with 
$m\equiv \pm 1\pmod 3$ ($\pm$ depends on
$\eta,\eta'$ individually)
already provide us with a number of remarkable examples. 
In this paper, we content ourselves with showing 
the following few cases among them.

\begin{Example}[Figure eight cevian orbit]
Let $\Delta=(0,i,-i)$ be a degenerate triangle, and set
$\eta(t)=-2e^{2\pi i t}+e^{-4\pi it}$, 
$\eta'(t)=2e^{-2\pi i t}+e^{4\pi it}$.
Then, the vertices of 
$\Delta(t)=\cS[\eta(t),\eta'(t)](\Delta)$ moves on a single 
figure eight curve $X^2=\frac{3}{16}X^4+Y^2$ in XY-coordinates of $\C$:
The first vertex of $\Delta(t)$ moves along
$\frac43\sqrt{3}\cos(t)+i \frac23\sqrt{3}\sin(2t)$
$(t\in \R)$ and the other two vertices chase it on the same orbit
(Figure \ref{Fig10}).

\begin{figure}[h]
\begin{center}
\begin{tabular}{ccc}
\ig{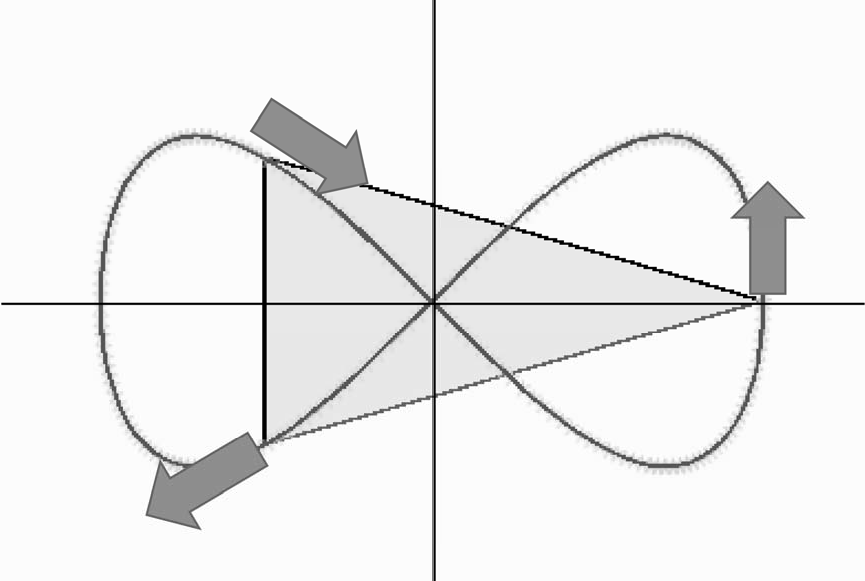} & 
$\to$ 
\ig{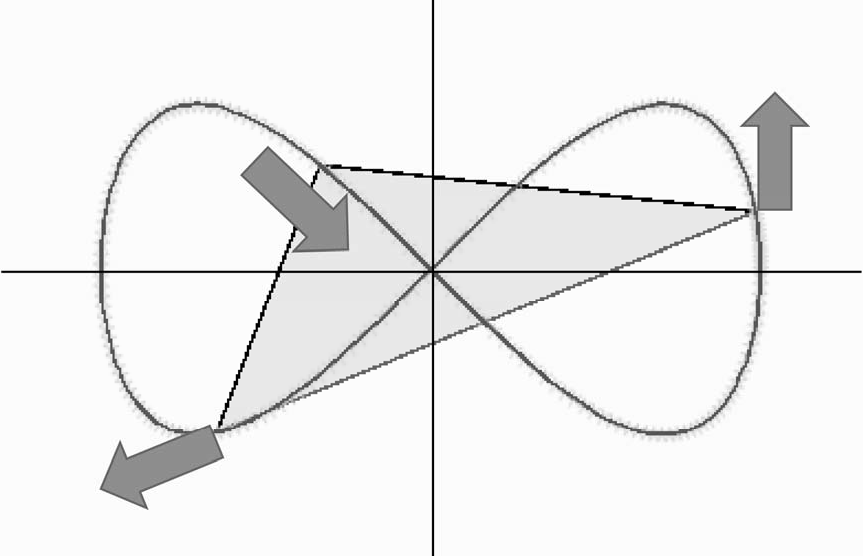} &
$\to$
\ig{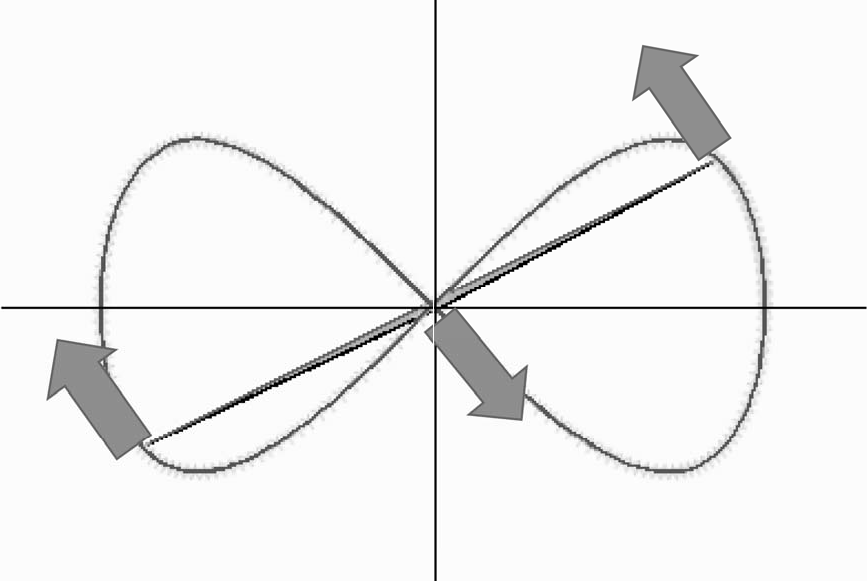} \\
$\to$
\ig{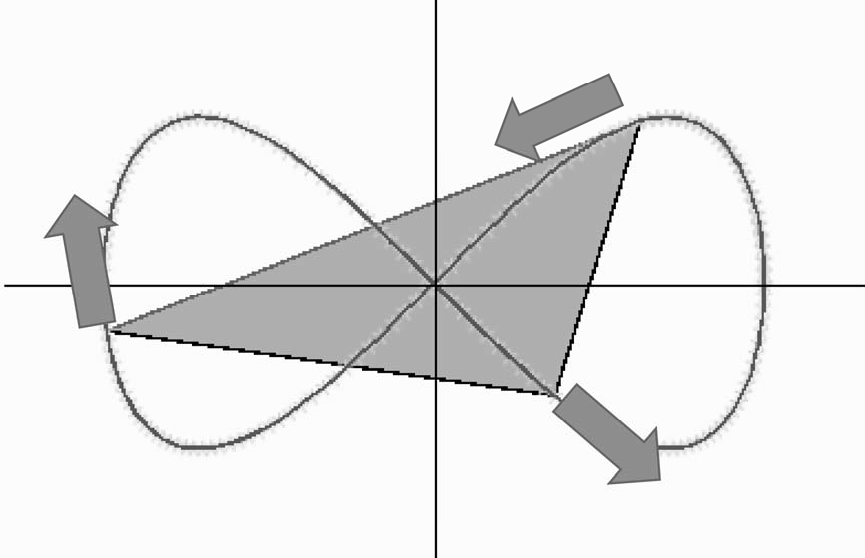} &
$\to$ 
\includegraphics[scale=0.17]{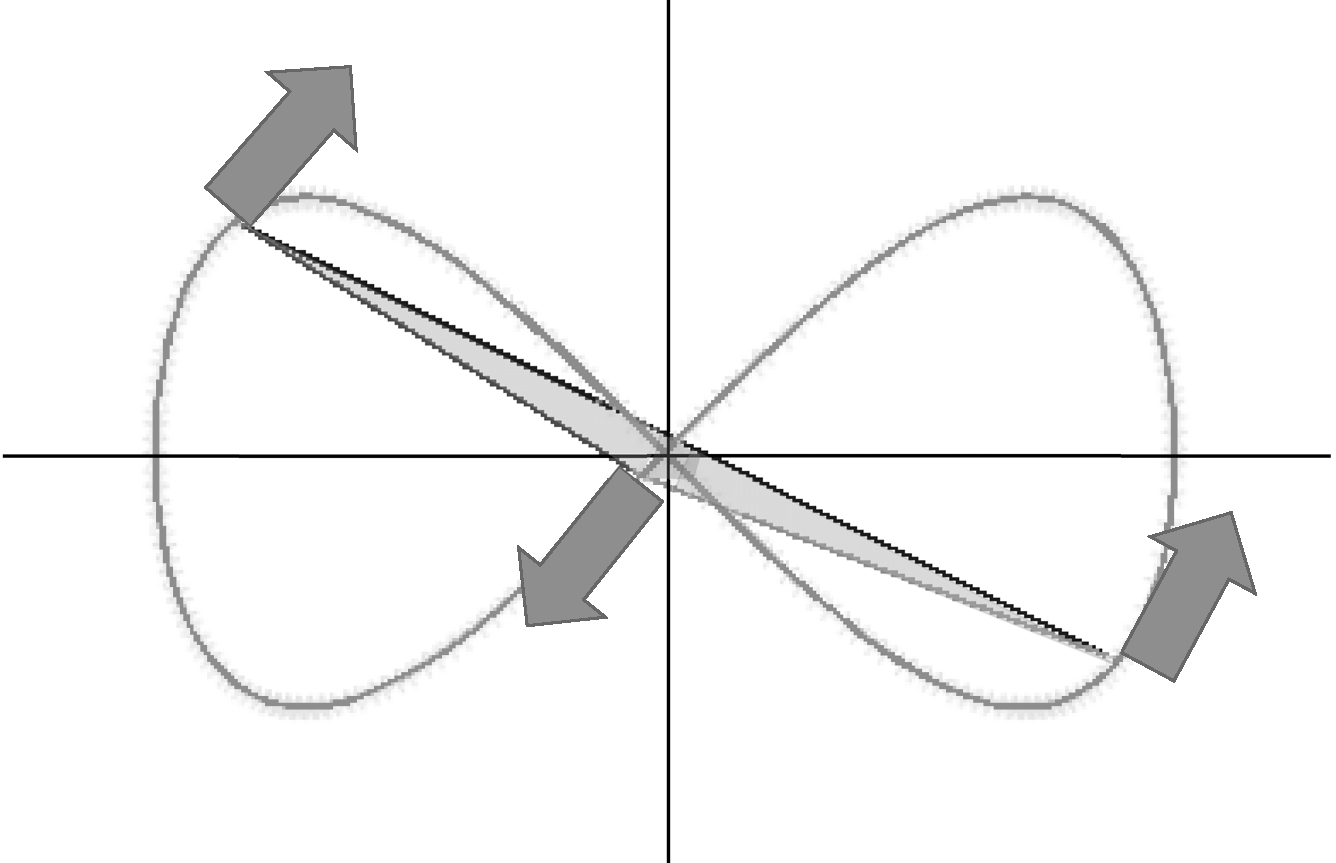} 
&
$\to$ 
\ig{greyfigure8ini.png} 
\end{tabular}
\caption{
$\bigl\{\cS[-2e^{2\pi i t}+e^{-4\pi it}, 2e^{-2\pi i t}+e^{4\pi it}](\Delta)
\bigr\}_{t\in\R}$ 
}
\label{Fig10}
\end{center}
\end{figure}
\noindent
One direction of generalizing this example is to 
consider $\Delta(t)$ whose vertices move on 
what is called a Lissajous curve.
In a separate article \cite{KNO20}, we will discuss in details 
``Lissajous 3-braids'' arising from this sort of triangle's motions.
\end{Example}

\begin{Example}[Figure eight median orbit]
We provide another example for Proposition 
\ref{singleorbit} (ii).
Starting from $\Delta=(0,4,3+i)$, the median triangle family
$\{\cM^{01/01}[\eta(t),\eta'(t)](\Delta)\}_{t\in\R}$ along 
with 
$\eta(t)=-2e^{2\pi i t}+e^{-4\pi it}+\omega$, 
$\eta'(t)=2e^{-2\pi i t}+e^{4\pi it}+\omega^2$
gives a figure eight orbit (Figure \ref{Fig11}).

\begin{figure}[h]
\centering
\begin{tabular}{cccc}
\igs{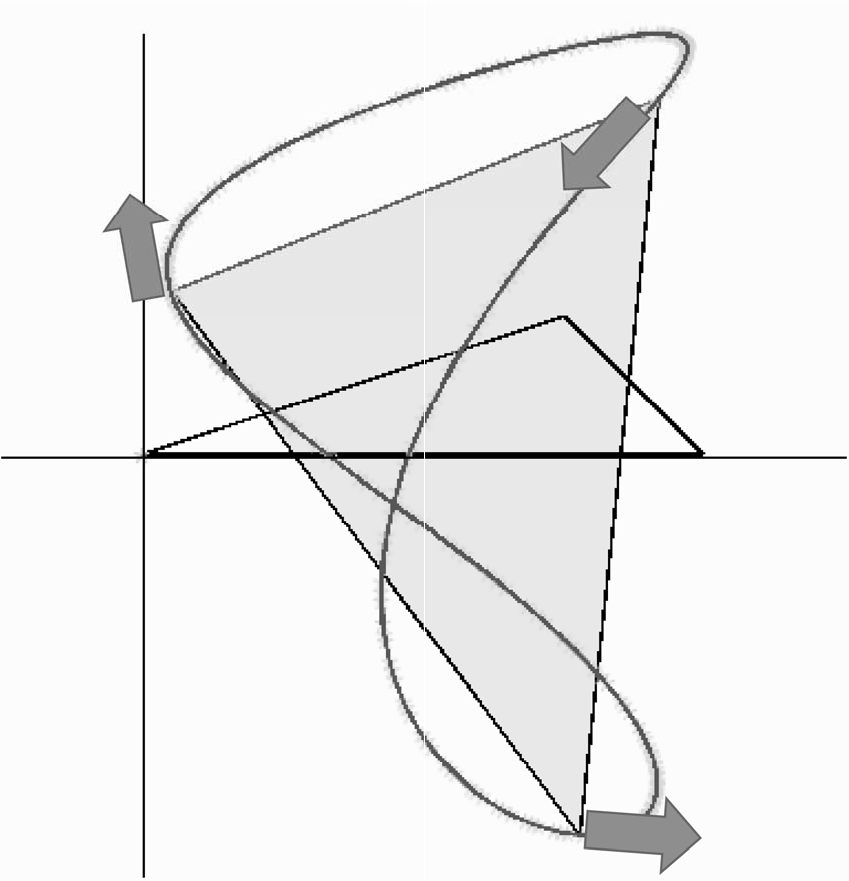} & 
$\to$ 
\igs{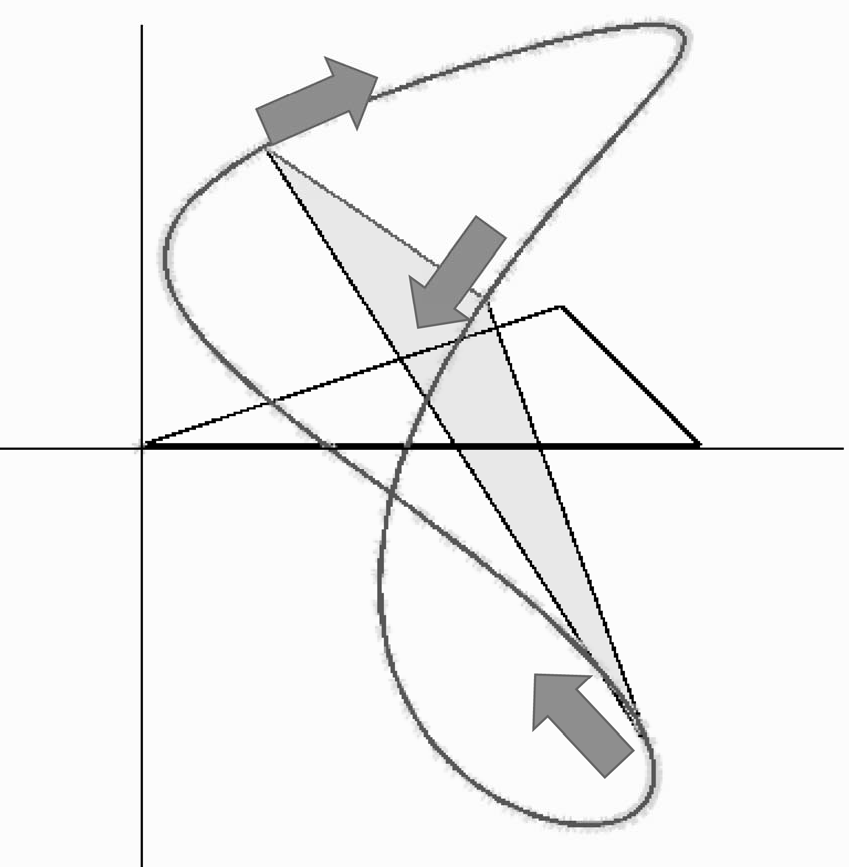} &
$\to$
\igs{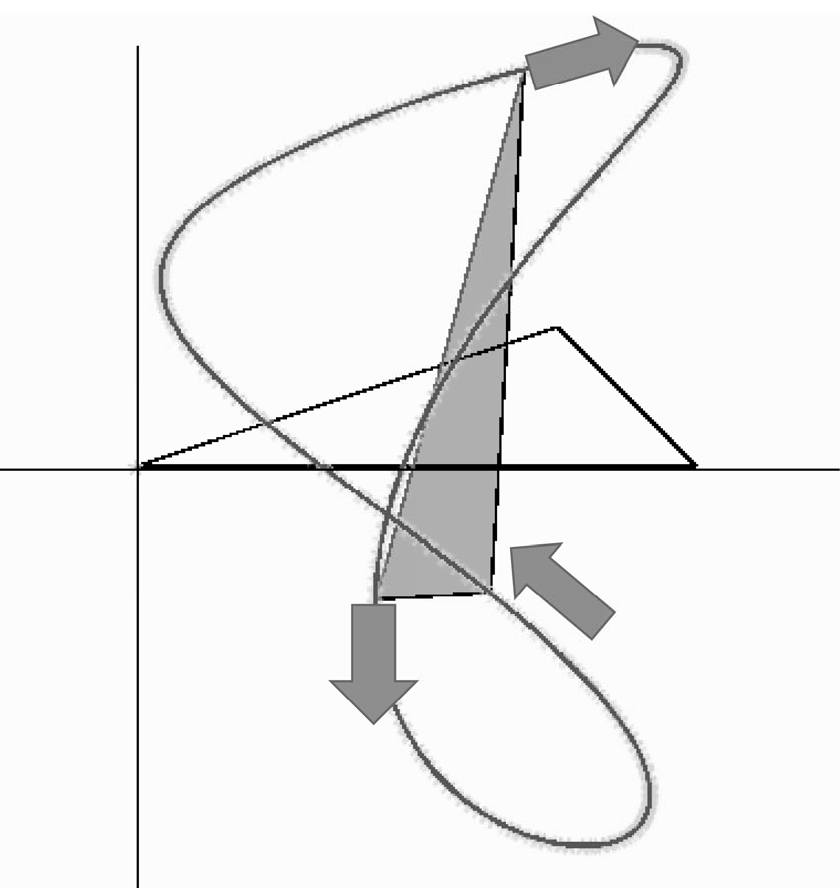} &
$\to$ 
\igs{greyMedian8ini.png}
\\
\end{tabular}
\vspace*{0.5cm} 
\caption{$\Delta$ and
$\bigl\{\cM^{01/01}[-2e^{2\pi i t}+e^{-4\pi it}+\omega,
2e^{-2\pi i t}+e^{4\pi it}+\omega^2](\Delta)
\bigr\}_{t\in\R}$ 
}
\label{Fig11}
\end{figure}
\end{Example}

\section*{Acknowledgement}
This work was supported by JSPS KAKENHI Grant Numbers JP16K13745, JP20H00115.
This is a pre-print of an article published in Results in Mathematics. The final authenticated version is available online at: https://doi.org/s00025-020-01268-3

\ifx\undefined\bysame
\newcommand{\bysame}{\leavevmode\hbox to3em{\hrulefill}\,}
\fi

\end{document}